\documentclass[preprint]{elsarticle}
\usepackage{natbib}
\usepackage{a4wide,graphicx,fancyhdr,amsmath,amssymb,mathtools,yfonts}
\usepackage[all]{xy}
\usepackage[utf8]{inputenc}
\usepackage{amsthm}
\usepackage[english]{babel}
\usepackage{chngcntr}
\usepackage{ifthen}
\usepackage{calc}

\numberwithin{equation}{section}
\newtheorem{theorem}{Theorem}[section]
\newtheorem{lemma}[theorem]{Lemma}

\newcommand{\N}{\mathbb{N}}
\newcommand{\Z}{\mathbb{Z}}
\newcommand{\Q}{\mathbb{Q}}

\newcounter{Constant}[chapter]
\newcounter{constant}[chapter]

\newcommand{\Cst}[1]
{
\ifthenelse{\value{#1}=0}
  {\addtocounter{Constant}{1}\setcounter{#1}{\value{Constant}}f_{\arabic{#1}}}
  {f_{\arabic{#1}}}
}

\newcommand{\cst}[1]
{
\ifthenelse{\value{#1}=0}
  {\addtocounter{constant}{1}\setcounter{#1}{\value{constant}}c_{\arabic{#1}}}
  {c_{\arabic{#1}}}
}

\begin{document}


\title{The Catalan Equation}
\author{P.H. Koymans\corref{cor1}}
\ead{p.h.koymans@math.leidenuniv.nl}
\address{Leiden University, Mathematical Institute, P.O. Box 9512, 2300 RA Leiden, The Netherlands}

\cortext[cor1]{Corresponding author}

\begin{abstract}
We consider the Catalan equation $x^p - y^q = 1$ in unknowns $x, y, p, q$, where $x, y$ are taken from an integral domain $A$ of characteristic $0$ that is finitely generated as a $\Z$-algebra and $p, q > 1$ are integers. We give explicit upper bounds for $p$ and $q$ in terms of the defining parameters of $A$. Our main theorem is a more precise version of a result of Brindza \citep{Brindza}. Brindza \citep{B} also gave inexplicit bounds for $p$ and $q$ in the special case that $A$ is the ring of $S$-integers for some number field $K$. As part of the proof of our main theorem, we will give a less technical proof for this special case with explicit upper bounds for $p$ and $q$.
\end{abstract}

\begin{keyword}
Catalan equation \sep Linear forms in logarithms
\MSC 11D61 \sep 11J86
\end{keyword}

\maketitle



\section{Introduction}
In 1844, Catalan conjectured that $8$ and $9$ are the only consecutive positive integers which both are perfect powers. More formally, the only solution in the natural numbers of
\begin{align}
\label{cati}
x^p - y^q = 1
\end{align}
for $p, q > 1$, $x, y > 0$ is $x = 3$, $p = 2$, $y = 2$, $b = 3$. Cassels \citep{Cassels} made the weaker conjecture that (\ref{cati}) has only finitely many solutions in positive integers $p, q > 1$, $x, y > 0$. Latter conjecture was proven by Tijdeman \citep{Tijdeman}. His proof heavily relies on the theory of linear forms in logarithms. A key point of Tijdeman's proof is that it is effective in the sense that an upper bound for the solutions can be computed.

Despite Tijdeman's work, Catalan's conjecture remained unproven until 2002. The problem was that the bounds resulting from Tijdeman's work were exceedingly large. In 2002, Mih\u ailescu \citep{Mih} was able to prove Catalan's conjecture using algebraic methods and avoiding linear forms in logarithms estimates.

Here, we consider Catalan's equation over other integral domains. Together with Brindza and Gy\H ory, Tijdeman was able to generalize his proof to the ring of integers of a number field $K$, see \citep{BGT}. They showed that there exists an effectively computable number $C$ which depends only on $K$ such that all solutions of the equation
\[
x^p \pm y^q = 1 \text{ in } x, y \in \mathcal{O}_K, p, q \in \N
\]
with $x, y$ not roots of unity and $p, q > 1$, $pq > 4$ satisfy
\[
\max(h(x), h(y), p, q) < C,
\]
where $h(\cdot)$ denotes the absolute logarithmic height of an algebraic number.

Brindza \citep{B} further generalized this to the ring of $S$-integers of a number field. However, Brindza's proof is quite technical. In section \ref{Brindza} we will prove Brindza's result by generalizing the proof given for the ordinary ring of integers in \citep{BGT}. Furthermore, we will make the resulting upper bounds for the solutions more explicit leading to Theorem \ref{tCatalana}.

From now on $c_1, c_2, \ldots$ are effectively computable constants depending only on $K$ and $S$. We use the notation $O(\cdot)$ as an abbreviation for $c$ times the expression between the parentheses, where $c$ is an effectively computable absolute constant. At each occurrence of $O(\cdot)$, the value of $c$ may be different. 

\begin{theorem}
\label{tCatalana}
Suppose that $p$ and $q$ are prime. Then all solutions of (\ref{catalan}) satisfy
\begin{align}
\label{pqprimebound}
\max\{p, q\} < (P^2s)^{O(Ps)} |D_K|^{6P} P^{P^2} =: \cst{11}
\end{align}
and
\begin{align}
\label{heightbound}
\max\{h(x), h(y)\} < (\cst{11} s)^{\cst{11}^6} |D_K|^{\cst{11}^4} Q^{\cst{11}^4}.
\end{align}
Furthermore, if $p$ and $q$ are arbitrary natural integers, we have
\begin{align}
\label{pqbound}
\max\{p, q\} < (\cst{11} s)^{\cst{11}^6} |D_K|^{\cst{11}^4} Q^{\cst{11}^4}.
\end{align}
\end{theorem}

Brindza \citep{Brindza} also gave effective upper bounds for $p$ and $q$ for the Catalan equation over finitely generated domains in the case that $x$ and $y$ are transcendental. In section \ref{Specialization} we will strengthen his result by giving explicit upper bounds for $p$ and $q$ without restrictions on $x$ and $y$. This will be our main theorem, which we state below.

Let $A = \Z[z_1, \ldots, z_r]$ be an integral domain finitely generated over $\Z$ with $r > 0$ and denote by $K$ the quotient field of $A$. We have
\[
A \cong \Z[X_1, \ldots, X_r]/I
\]
where $I$ is the ideal of polynomials $f \in \Z[X_1, \ldots, X_r]$ such that $f(z_1, \ldots, z_r) = 0$. Then $I$ is finitely generated. Let $d \geq 1$, $h \geq 1$ and assume that
\[
I = (f_1, \ldots, f_m)
\]
with $\deg f_i \leq d$, $h(f_i) \leq h$ for $i = 1, \ldots, m$. Here $\deg$ means the total degree of the polynomial $f_i$ and $h(f_i)$ is the logarithmic height of $f_i$., i.e., the logarithm of the maximum of the absolute values of the coefficients of $f_i$.

\begin{theorem}
\label{tMT}
All solutions of the equation
\[
x^p - y^q = 1
\]
in positive integers $p$ and $q$, $x, y \in A$ and $x, y$ not roots of unity must satisfy
\[
\label{resulta}
\max\{p, q\} < (2d)^{\exp O(r)}
\]
if $x, y$ are transcendental and
\[
\label{resultt}
\max\{p, q\} < \exp \exp \exp\left((2d)^{\exp O(r)} (h + 1)\right)
\]
if $x, y$ are algebraic.
\end{theorem}

In the case that $x$ and $y$ are transcendental, we will use a relatively straightforward function field argument. But in the case that $x$ and $y$ algebraic presents more difficulties. The proof uses a specialization technique similar to that in \cite{EG}. By means of a so called specialization homomorphism we embed our finitely generated domain into an algebraic number field, after which we can apply our theorem \ref{tCatalana}.

\section{Preliminaries}
This section contains some preliminaries about function fields and number fields.

\subsection{Function fields}
Let $k$ be a field. A function field $K$ over $k$ is a finitely generated field extension of transcendence degree $1$ over $k$. From now on we will assume that $k$ is algebraically closed and of characteristic $0$. Denote by $M_K$ the set of normalized discrete valuations on $K$ that are trivial on $k$. These satisfy the so called sum formula
\[
\sum_{v \in M_K} v(x) = 0
\]
for $x \in K^\ast$. Let $\mathbf{x} = (x_1, \ldots, x_n) \in K^n \setminus \{0\}$ be a vector. We define
\[
v(\mathbf{x}) := -\min(v(x_1), \ldots, v(x_n)) \text{ for } v \in M_K
\]
and
\[
H_K^{\text{hom}}(\mathbf{x}) = H_K^{\text{hom}}(x_1, \ldots, x_n) := \sum_{v \in M_K} v(\mathbf{x}).
\]
We call $H_K^{\text{hom}}(\mathbf{x})$ the homogeneous height of $\mathbf{x}$ with respect to $K$. Let $L$ be a finite extension of $K$. Then
\[
H_L^{\text{hom}}(\mathbf{x}) = [L : K] H_K^{\text{hom}}(\mathbf{x}).
\]
Next we define the height for elements of $K$ by
\[
H_K(x) := H_K^{\text{hom}}(1, x) = -\sum_{v \in M_K} \min(0, v(x)).
\]
Now we mention the most important properties of the height $H_K$. It is well known that
\[
H_K(x) \geq 0 \text{ for } x \in K, \ H_K(x) = 0 \Leftrightarrow x \in k.
\]
Furthermore, it follows from the sum formula that
\[
H_K(x^m) = |m|H_K(x) \text{ for } x \in K^\ast, m \in \Z,
\]
\[
H_K(x + y) \leq H_K(x) + H_K(y),
\]
and
\[
H_K(xy) \leq H_K(x) + H_K(y)
\]
for $x, y \in K$. We conclude that
\[
H_K(x) = \frac{1}{2} \left(H_K(x) + H_K(x^{-1})\right) = \frac{1}{2} \sum_{v \in M_K} |v(x)| \geq \frac{1}{2}|S| \text{ for } x \in K^\ast, 
\]
where $S$ is the set of valuations $v \in M_K$ for which $v(x) \neq 0$.

Let $S$ be a finite subset of $M_K$. Then the group of $S$-units of $K$ is given by
\[
\mathcal{O}_S^\ast = \{x \in K^\ast : v(x) = 0 \text{ for } v \in M_K \setminus S\}.
\]
We denote by $g_{K/k}$ the genus of $K$ over $k$.

\begin{theorem}
\label{tMason}
Let $K$ be a finite extension of $k(z)$ and $S$ be a finite subset of $M_K$. Then for every solution of
\[
x + y = 1 \text{ in } x, y \in \mathcal{O}_S^\ast \setminus k^\ast
\]
we have
\[
\max(H_K(x), H_K(y)) \leq |S| + 2g_{K/k} - 2.
\]
\end{theorem}

\begin{proof}
See Chapter I, section 3, Lemma 2 of Mason \citep{Mason}.
\end{proof}

To apply this theorem, we need an upper bound for the genus. Such an upper bound is provided by the following lemma.

\begin{lemma}
\label{lSchmidt}
Let $K$ be the splitting field over $k(z)$ of $F := X^m + f_1X^{m - 1} + \cdots + f_m$, where $f_1, \ldots, f_m \in k[z]$. Then
\[
g_{K/k} \leq (d - 1)m \cdot \max_{1 \leq i \leq m} \deg f_i,
\]
where $d = [K : k(z)]$.
\end{lemma}

\begin{proof}
This is lemma H of Schmidt \citep{Schmidt}.
\end{proof}

\subsection{Algebraic number fields}
Let $K$ be an algebraic number field with ring of integers $\mathcal{O}_K$. We introduce a collection of absolute values $\{|\cdot|_v\}$ on $K$. A real place of $K$ is a set $\{\sigma\}$ where $\sigma: K \rightarrow \mathbb{R}$ is a real embedding of $K$. A complex place of $K$ is a pair $\{\sigma, \overline{\sigma}\}$ of conjugate complex embeddings $K \rightarrow \mathbb{C}$. An infinite place is a real or complex place. A finite place of $K$ is a non-zero prime ideal of $\mathcal{O}_K$. Denote by $M_K$ the set of all places of $K$.

We associate to every place $v \in M_K$ an absolute value $|\cdot|_v$, which we define as follows for $\alpha \in K$:
\begin{align*}
|\alpha|_v &:= |\sigma(\alpha)| \text{ if } v = \{\sigma\} \text{ is real;} \\
|\alpha|_v &:= |\sigma(\alpha)|^2 = |\overline{\sigma}(\alpha)|^2 \text{ if } v = \{\sigma, \overline{\sigma}\} \text{ is complex;} \\
|\alpha|_v &:= N(\mathfrak{p})^{-\text{ord}_\mathfrak{p}(\alpha)} \text{ if } v = \mathfrak{p} \text{ is a prime ideal of } \mathcal{O}_K, \text{ where } N(\mathfrak{p}) := |\mathcal{O}_K/\mathfrak{p}|.
\end{align*}
Then we have the product formula over $K$
\[
\prod_{v \in M_K} |\alpha|_v = 1
\]
for $\alpha \in K^\ast$. Later on it will be useful to deal with all absolute values simultaneously. For this we have the useful inequality
\[
|x_1 + \ldots + x_n| \leq n^{s(v)} \max(|x_1|_v, \ldots, |x_n|_v)
\]
for $v \in M_K$, $x_1, \ldots, x_n \in K$, where $s(v) = 1$ if $v$ is real, $s(v) = 2$ if $v$ is complex and $s(v) = 0$ if $v$ is finite. Furthermore, $|\alpha|_v^{1/2}$ satisfies the triangle inequality for all $v \in M_K$.

Let $S$ denote a finite subset of $M_K$ containing all infinite places. Write $s = |S|$. We define the ring of $S$-integers by
\[
\mathcal{O}_S := \{\alpha \in K: |\alpha|_v \leq 1 \text{ for all } v \in M_K \setminus S\}.
\]
This is a subring of $K$ containing $\mathcal{O}_K$, hence it is a Dedekind domain. Concretely, this means that every non-zero ideal of $\mathcal{O}_S$ factors uniquely into prime ideals.

Let $W_K$ denote the group of roots of unity of $K$. Then we have the following important generalization of the well-known Dirichlet's unit theorem.
\begin{theorem}
\label{tDirichlet}
We have
\[
\mathcal{O}_S^\ast \cong W_K \times \Z^{s - 1}.
\]
More explicitly, there are $\varepsilon_1, \ldots, \varepsilon_{s - 1} \in \mathcal{O}_S^\ast$ such that every $\varepsilon \in \mathcal{O}_S^\ast$ can be expressed uniquely as
\[
\varepsilon = \zeta \varepsilon_1^{b_1} \cdots \varepsilon_{s - 1}^{b_{s - 1}},
\]
where $\zeta$ is a root of unity of $K$ and $b_1, \ldots, b_{s - 1}$ are rational integers.
\end{theorem}

\begin{proof}
See page 104 in \citep{Lang}.
\end{proof}

A system $\{\varepsilon_1, \ldots, \varepsilon_{s - 1}\}$ as above is called a fundamental system of $S$-units. Write $S = \{v_1, \ldots, v_s\}$. We define the $S$-regulator by
\[
R_S := \left|\det\left(\log|\varepsilon_i|_{v_j} \right)_{i, j = 1, \ldots, s - 1} \right|.
\]
Then $R_S \neq 0$ and furthermore $R_S$ is independent of the choice of $\varepsilon_1, \ldots, \varepsilon_{s - 1}$ and of the choice $v_1, \ldots, v_{s - 1}$ of $S$. Then we have by Lemma 3 in \citep{BG}
\begin{align}
\label{sregl}
R_S \geq 0.2052 (\log 2)^t,
\end{align}
where we recall that $t$ is the number of finite places of $S$.

We define the absolute multiplicative height of $\alpha \in K$ by
\[
H(\alpha) := \prod_{v \in M_K} \max(1, |\alpha|_v)^{1/[K: \Q]}.
\]
Next we define the absolute logarithmic height by
\[
h(\alpha) := \log H(\alpha).
\]
Let $\alpha, \alpha_1, \ldots, \alpha_n \in K$ and $m \in \Z$. Then we have the following important properties
\begin{align*}
&h(\alpha_1 \cdots \alpha_n) \leq \sum_{i = 1}^n h(\alpha_i); \\
&h(\alpha_1 + \cdots + \alpha_n) \leq \log n + \sum_{i = 1}^n h(\alpha_i); \\
&h(\alpha^m) = |m| h(\alpha) \text{ if } \alpha \neq 0.
\end{align*}
For a proof of the above properties, see chapter 3 in \citep{Waldschmidt}. Furthermore, we have Northcott's theorem. 

\begin{theorem}
\label{tNorthcott}
Let $D, H$ be positive reals. Then there are only finitely many $\alpha \in \overline{\Q}$ such that $\deg \alpha \leq D$ and $h(\alpha) \leq H$.
\end{theorem}

\begin{proof}
See Theorem 1.9.3 in \citep{Book}.
\end{proof}

\section{Lemmas}
In this section we will formulate the necessary lemmas. This section is subdivided into three subsections. In the first subsection we will give some algebraic lemmas. In the second and third subsection we cover advanced lemmas concerning linear forms in logarithms and the hyperelliptic equation.

Let $K$ be a number field of degree $d$, discriminant $D_K$ and denote by $M_K$ the set of places of $K$. Let $S$ be a finite subset of $M_K$ containing all infinite places. Write $s = |S|$. Let $\mathfrak{p}_1, \ldots, \mathfrak{p}_t$ be the prime ideals in $S$. Put
\[
P := \max\{2, N(\mathfrak{p}_1), \ldots, N(\mathfrak{p}_t)\}
\]
and
\[
Q := N(\mathfrak{p}_1 \cdots \mathfrak{p}_t).
\]

\subsection{Algebraic lemmas}
Our first lemma gives a lower bound for the height of $\alpha \in K$.

\begin{lemma}
\label{lVoutier}
Let $\alpha \in K$, $\alpha \neq 0$, $\alpha$ not a root of unity. Then
\begin{align}
dh(\alpha) \geq \frac{\log 2}{(\log(3d))^3} =: \cst{1}.
\end{align}
\end{lemma}

\begin{proof}
This follows from the work in \citep{PMV}.
\end{proof}

Now we need some results on $S$-units. Lemma \ref{lSunit} is an effective version of Theorem \ref{tDirichlet}.

\begin{lemma}
\label{lSunit}
There is a fundamental system of $S$-units $\{\eta_1, \ldots ,\eta_{s-1}\}$ such that
\begin{enumerate}
\item[(i)] $\prod_{i=1}^{s-1} h(\eta_i) \leq (2s)^{O(s)}R_S$,
\item[(ii)] $h(\eta_i)\leq (2s)^{O(s)}R_S$ for $i = 1, \ldots ,s - 1$,
\item[(iii)] the absolute values of the entries of the inverse of the matrix $(\log |\eta_i|_{v_j})_{i,j = 1, \ldots ,s - 1}$ do not exceed $(2s)^{O(s)}$.
\end{enumerate}
\end{lemma}

\begin{proof}
This is a less precise version of Lemma 1 in \citep{BG}.
\end{proof}

Let $h$ denote the class number of $K$, let $r$ be the unit rank and let $R$ be the regulator of $K$. Put
\[
\cst{2} :=
\left\{
	\begin{array}{ll}
		0  & \mbox{if } r = 0, \\
		1/d & \mbox{if } r = 1, \\
		29er! r \sqrt{r - 1} \log d & \mbox{if } r \geq 2.
	\end{array}
\right.
\]
Define the $S$-norm of $\alpha \in K$ by $N_S(\alpha) := \prod_{v \in S} |\alpha|_v$.

\begin{lemma}
\label{lhsmall}
Let $\alpha \in \mathcal{O}_K \setminus \{0\}$ and let $n$ be a positive integer. Then there exists $\varepsilon \in \mathcal{O}_S^\ast$ such that
\[
h(\varepsilon^n \alpha) \leq \frac{1}{d} \log N_S(\alpha) + n\left(\cst{2} R + \frac{h}{d} \log Q\right).
\]
\end{lemma}

\begin{proof}
See Proposition 4.3.12 in \citep{Book}.
\end{proof}

Let $\alpha, \beta \in K^\ast$. Put
\[
H := \max\{1, h(\alpha), h(\beta)\}.
\]
We define $\log^\ast x := \max(1, \log x)$ for any real number $x > 0$.

\begin{lemma}
\label{lSuniteq}
Every solution $x$, $y$ of
\[
\alpha x + \beta y = 1 \text{ in } x, y \in \mathcal{O}_S^\ast
\]
satisfies
\[
\max(h(x), h(y)) < (2s)^{O(s)} (P/ \log P) H R_S \max\{\log P, \log^\ast R_S\}.
\]
\end{lemma}

\begin{proof}
This is a less precise version of Theorem 1 in \citep{GY}.
\end{proof}

\subsection{Linear forms in logarithms}
Let $K$ be an algebraic number field of degree $d$, and assume that it is embedded in $\mathbb{C}$. We put $\chi = 1$ if $K$ is real, and $\chi = 2$ otherwise. Let
\[
\Sigma = b_1 \log \alpha_1 + \cdots + b_n \log \alpha_n
\]
where $\alpha_1, \ldots, \alpha_n$ are $n (\geq 2)$ non-zero elements of $K$ with some fixed non-zero values of $\log \alpha_1, \ldots, \log \alpha_n$, and $b_1, \ldots, b_n$ are rational integers, not all zero. We put
\[
A_i \geq \max\{dh(\alpha_i), |\log \alpha_i|, 0.16\}, i = 1, \ldots, n
\]
and
\[
B \geq \max\{1, \max\{|b_i|A_i/A_n : 1 \leq i \leq n\}\}.
\]

\begin{theorem}
\label{tMatveev}
Suppose that $\Sigma \neq 0$. Then
\begin{align*}
\log |\Sigma| > -a_1(n, d)A_1 \cdots A_n \log(eB),
\end{align*}
where
\[
a_1(n, d) = \min\left\{\frac{1}{\chi} \left(\frac{1}{2} en\right)^\chi 30^{n + 3} n^{3.5}, 2^{6n + 20} \right\} d^2 \log(ed).
\]
Further, $B$ may be replaced by $\max(|b_1|, \ldots, |b_n|)$.
\end{theorem}

\begin{proof}
This is Corollary 2.3 of \citep{Matveev}. 
\end{proof}

Put
\begin{align}
\label{lambdadef}
\Lambda = \alpha_1^{b_1} \cdots \alpha_n^{b_n} - 1
\end{align}
and
\[
A_i' = dh(\alpha_i) + \pi, \quad i = 1, \ldots, n.
\]

\begin{lemma}
\label{lMatveev}
Suppose that $\Lambda \neq 0$, and that $B'$ satisfies
\[
B' \geq \max\{1, \max\{|b_i|A_i'/A_n' : 1 \leq i \leq n\}\}.
\]
Then we have
\[
\log |\Lambda| > -a_2(n, d) A_1' \cdots A_n' \log\left(e (n + 1) B'\right),
\]
where
\[
a_2(n, d) = 2\pi \min\left\{\frac{1}{\chi} \left(\frac{1}{2} e(n + 1)\right)^{\chi} 30^{n + 4} (n + 1)^{3.5}, 2^{6n + 26}\right\} d^2 \log(ed).
\]
Again $B$ may be replaced by $\max(|b_1|, \ldots, |b_n|)$.
\end{lemma}

\begin{proof}
We use the principal value of the logarithm. Let $z$ be a complex number such that $|z - 1| < \frac{1}{2}$. Then
\[
|\log(z)| = \left|\sum_{n = 1}^\infty (-1)^{n - 1} (z - 1)^n\right| \leq |z - 1| \sum_{n = 1}^\infty |z - 1|^{n - 1} = |z - 1| \frac{1}{1 - |z - 1|} < 2|z - 1|.
\]
Hence
\[
|z - 1| > \frac{1}{2}|\log(z)|.
\]
We apply this with $z = \alpha_1^{b_1} \cdots \alpha_n^{b_n}$. Because we want to give a lower bound for $\left|z - 1\right|$, we may assume that $\left|z - 1\right| < \frac{1}{2}$. This gives
\[
\left|z - 1\right| > \frac{1}{2} \left|\log \left(\alpha_1^{b_1} \cdots \alpha_n^{b_n} \right) \right| = \frac{1}{2} \left|b_1 \log(\alpha_1) + \cdots + b_n \log(\alpha_n) + 2k\pi i \right|
\]
for some $k \in \Z$. But
\[
\left|z - 1\right| < \frac{1}{2},
\]
so taking imaginary parts
\[
|k| \leq \frac{1}{2\pi} \left(1 + |b_1|\pi + \cdots + |b_n|\pi \right) \leq \frac{n + 1}{2} \max\{|b_1|, \ldots, |b_n|\}.
\]
Put
\[
\Sigma = b_1 \log(\alpha_1) + \cdots + b_n \log(\alpha_n) + 2k\pi i = b_1 \log(\alpha_1) + \cdots + b_n \log(\alpha_n) + 2k\log(-1).
\]
We apply Theorem \ref{tMatveev} with $n + 1$, $(-1, \alpha_1, \ldots, \alpha_n)$ and $(2k, b_1, \ldots, b_n)$. Then
\[
|\log \alpha_i| \leq \log |\alpha_i| + \pi \leq dh(\alpha_i) + \pi.
\]
So we can take $A_1 = \pi$, $A_i = A_i'$ for $i = 2, \ldots, n + 1$ and $B = (n + 1)B'$. Our assumption $\Lambda \neq 0$ implies $\Sigma \neq 0$. Theorem \ref{tMatveev} gives
\begin{align*}
\left|z - 1\right| &> \frac{1}{2} \left|b_1 \log(\alpha_1) + \cdots + b_n \log(\alpha_n) + 2k\pi i \right| \\
&> \frac{1}{2} \exp(-a_1(n + 1, d) A_1 \cdots A_{n + 1} \log(e (n + 1)B'))
\end{align*}
where
\[
a_1(n, d) = \min\left\{\frac{1}{\chi} \left(\frac{1}{2} en\right)^{\chi} 30^{n + 3} n^{3.5}, 2^{6n + 20}\right\} d^2 \log(ed).
\]
This implies
\[
\log \left|z - 1\right| > -a_2(n, d) A_1' \cdots A_n' \log\left(e (n + 1) B'\right)
\]
where
\[
a_2(n, d) = 2\pi \min\left\{\frac{1}{\chi} \left(\frac{1}{2} e(n + 1)\right)^{\chi} 30^{n + 4} (n + 1)^{3.5}, 2^{6n + 26}\right\} d^2 \log(ed)
\]
as desired.
\end{proof}

Keep the above notation and assumptions and consider again $\Lambda$ as defined by (\ref{lambdadef}). Let now $B$ and $B_n$ be real numbers satisfying
\[
B \geq \max\{|b_1|, \ldots, |b_n|\}, B \geq B_n \geq |b_n|.
\]
Let $\mathfrak{p}$ be a prime ideal of $\mathcal{O}_K$ and denote by $e_\mathfrak{p}$ and $f_\mathfrak{p}$ the ramification index and the residue class degree of $\mathfrak{p}$, respectively. Suppose that $\mathfrak{p}$ lies above the rational prime number $p$. Then $N(\mathfrak{p}) = p^{f_\mathfrak{p}}$.

\begin{lemma}
\label{lYu}
Assume that $\text{ord}_p b_n \leq \text{ord}_p b_i$ for $i = 1, \ldots, n$ and set
\[
h'_i := \max\{h(\alpha_i), 1/16e^2d^2\}, \quad i = 1, \ldots, n.
\]
If $\Lambda \neq 0$, then for any real $\delta$ with $0 < \delta \leq 1/2$ we have
\[
\text{ord}_\mathfrak{p} \Lambda < a_3(n, d) \frac{e^n_\mathfrak{p} N(\mathfrak{p})}{(\log N(\mathfrak{p}))^2} \max \left\{h'_1 \cdots h'_n \log(M\delta^{-1}), \frac{\delta B}{B_n a_4(n, d)}\right\},
\]
where
\begin{align*}
a_3(n, d) = (16ed)^{2(n + 1)} n^{3/2} \log(2nd) \log(2d), \\
a_4(n, d) = (2d)^{2n + 1} \log(2d) \log^3(3d),
\end{align*}
and
\[
M = B_n a_5(n, d) N(\mathfrak{p})^{n + 1} h'_1 \cdots h'_{n - 1}
\]
with
\[
a_5(n, d) = 2e^{(n + 1)(6n + 5)} d^{3n} \log(2d).
\]
\end{lemma}

\begin{proof}
This is the second consequence of the Main Theorem in \citep{Yu}.
\end{proof}

\subsection{The super- and hyperelliptic equation}
Let
\[
f(X) = a_0X^n + a_1X^{n - 1} + \cdots + a_0 \in \mathcal{O}_S[X]
\]
be a polynomial of degree $n \geq 2$ without multiple roots and let $b$ be a non-zero element of $\mathcal{O}_S$. Put
\[
\hat{h} := \frac{1}{d} \sum_{v \in M_K} \log \max(1, |b|_v, |a_0|_v, \ldots, |a_n|_v).
\]
Our next lemma concerns the superelliptic equation
\begin{align}
\label{superelliptic}
f(x) = by^m
\end{align}
in $x, y \in \mathcal{O}_S$ with a fixed exponent $m \geq 3$.

\begin{lemma}
\label{lsuper}
Assume that $m \geq 3$, $n \geq 2$. If $x, y \in \mathcal{O}_S$ is a solution to equation (\ref{superelliptic}) then we have
\[
h(x), h(y) \leq (6ns)^{14m^3n^3s} |D_K|^{2m^2n^2} Q^{3m^2n^2} e^{8m^2n^3d\hat{h}}.
\]
\end{lemma}

\begin{proof}
See Theorem 2.1 in \citep{online}.
\end{proof}

We now consider the hyperelliptic equation
\begin{align}
\label{hyperelliptic}
f(x) = by^2
\end{align}
in $x, y \in \mathcal{O}_S$.

\begin{lemma}
\label{lhyper}
Assume that $n \geq 3$. If $x, y \in \mathcal{O}_S$ is a solution to equation (\ref{hyperelliptic}) then we have
\[
h(x), h(y) \leq (4ns)^{212n^4s} |D_K|^{8n^3} Q^{20n^3} e^{50n^4d\hat{h}}.
\]
\end{lemma}

\begin{proof}
See Theorem 2.2 in \citep{online}.
\end{proof}

The following lemma is an explicit version of the Schinzel-Tijdeman theorem over the $S$-integers.

\begin{lemma} 
\label{lST}
Assume that (\ref{superelliptic}) has a solution $x, y \in \mathcal{O}_S$ where $y$ is neither 0 nor a root of unity. Then
\[
m \leq (10n^2s)^{40ns} |D_K|^{6n} P^{n^2} e^{11nd\hat{h}}.
\]
\end{lemma}

\begin{proof}
See Theorem 2.3 in \citep{online}.
\end{proof}

\section{Proof of Theorem \ref{tCatalana}}
\label{Brindza}
We will give effective bounds for the solutions of the Catalan equation over the ring of $S$-integers of a number field $K$. Such bounds, in an inexplicit form, were already obtained in \citep{B}. Below we obtain more precise bounds by a less technical argument. Instead of following \citep{B}, we generalize the proof in \citep{BGT} dealing with the Catalan equation for the ordinary ring of integers.

We start with some notation. Let $K$ be a number field of degree $d$ and discriminant $D_K$ and denote by $M_K$ the set of places of $K$. Let $S$ be a finite subset of $M_K$ containing all infinite places. Write $s = |S|$. Let $\mathfrak{p}_1, \ldots, \mathfrak{p}_t$ be the prime ideals in $S$. Put
\[
P := \max\{2, N(\mathfrak{p}_1), \ldots, N(\mathfrak{p}_t)\}
\]
and
\[
Q := N(\mathfrak{p}_1 \cdots \mathfrak{p}_t).
\]
Consider the equation
\begin{align}
\label{catalan}
x^p \pm y^q = 1
\end{align}
in $x, y \in \mathcal{O}_S$, $p, q \in \mathbb{N}$ with $x, y$ not roots of unity and $p, q > 1$, $pq > 4$.

In the course of our proof we use the following simple lemma.

\begin{lemma}
\label{lEstimates2}
Let $a > 0$, $b > 1$, $c > 0$ and $x > 0$. Assume that
\[
\frac{a}{\log b} c^{1/a} > e
\]
and
\[
b^{x/a} > 2\frac{a}{\log b} c^{1/a} \log\left(\frac{a}{\log b} c^{1/a}\right).
\]
Then
\[
\frac{x^a}{b^x} < c^{-1}.
\]
\end{lemma}

\begin{proof}
Take $z := b^{x/a} = e^{x \log b/a}$. Then
\begin{align*}
\frac{x^a}{b^x} < c^{-1} &\Leftrightarrow \frac{x}{z} < c^{-1/a} \Leftrightarrow \frac{\log z}{z} < \frac{\log b}{a} c^{-1/a} \\
&\Leftrightarrow \frac{z}{\log z} > \frac{a}{\log b} c^{1/a}.
\end{align*}
In general, if $A$, $z$ are reals with $A > e$, $z > A \log A$, then $\frac{z}{\log z} > A$. Applying this with $A := \frac{a}{\log b} c^{1/a}$ the lemma follows.
\end{proof}

\subsection{A key theorem}
Before proving Theorem \ref{tCatalana}, we generalize Lemma 6 in \citep{BGT}. The proof is a more modern and simplified version of Theorem 9.3 in \citep{ST}. Consider the equation
\begin{align}
\label{suniteq}
x_1 + x_2 = y^q
\end{align}
in $x_1, x_2 \in \mathcal{O}_S^\ast$, $y \in \mathcal{O}_S$ not zero and not an $S$-unit, and prime numbers $q$.

\begin{theorem}
\label{t8}
Equation (\ref{suniteq}) implies that
\[
q \leq (2s)^{O(s)} P^2 R_S^4.
\] 
\end{theorem}

Before proving Theorem \ref{t8}, we make some simplifications. We have the useful inequality
\[
d \leq 2s,
\]
which we will use throughout without further mention.

Choose a fundamental system of $S$-units $\{\eta_1, \ldots ,\eta_{s-1}\}$ as in Lemma \ref{lSunit}. We may write
\[
x_1 = \zeta_1 \eta_1^{a_1} \cdots \eta_{s - 1}^{a_{s - 1}}, \quad x_2 = \zeta_2 \eta_1^{b_1} \cdots \eta_{s - 1}^{b_{s - 1}}
\]
where $a_1, \ldots, a_{s - 1}, b_1, \ldots, b_{s - 1} \in \Z$ and $\zeta_1, \zeta_2 \in \mathcal{O}_K$ roots of unity. We assume
\begin{align}
\label{assq}
0 \leq b_i < q \text{ for } i = 1, \ldots, s - 1.
\end{align}
This is no loss of generality. Indeed, for $i = 1, \ldots, s - 1$, write
\[
b_i = qb_{i, 1} + b_{i, 2}, \quad 0 \leq b_{i, 2} < q,
\]
and
\[
\varepsilon_1 = \eta_1^{b_{1, 1}} \cdots \eta_{s - 1}^{b_{{s - 1}, 1}}, \quad \varepsilon_2 = \eta_1^{b_{1, 2}} \cdots \eta_{s - 1}^{b_{{s - 1}, 2}}.
\]
Thus $x_2 = \zeta_2 \varepsilon_2 \varepsilon_1^q$.  Then on replacing $x_1$, $x_2$, $y$ by $x_1 \varepsilon_1^{-q}$, $x_2 \varepsilon_1^{-q}$, $y \varepsilon_1^{-1}$, we get another solution of (\ref{suniteq}) with 
(\ref{assq}). Assuming henceforth (\ref{assq}), we put
\[
W := \max(|a_1|, \ldots, |a_{s - 1}|, b_1, \ldots, b_{s - 1}).
\]
We prove two lemmas before proving the key theorem. Take $\cst{53} := (2s)^{\Cst{52}s} P^2 R_S^2$ with $\Cst{52}$ a sufficiently large absolute constant.

\begin{lemma}
Assume (\ref{assq}) and assume also that $q > \cst{53}$. Then
\begin{align}
\label{lemma1}
W \leq \cst{54} q h(y)
\end{align}
with $\cst{54} := (2s)^{O(s)}R_S$.
\end{lemma}

\begin{proof}
By $\max(b_1, \ldots, b_{s - 1}) < q$, (\ref{sregl}) and Lemma \ref{lVoutier}, we may assume that
\[
W = \max(|a_1|, \ldots, |a_{s - 1}|).
\]
Fix $v \in S$. Then we have
\[
|x_1|_v = |y^q - x_2|_v \leq 4 \max(|y|^q_v, |x_2|_v).
\]
Hence
\[
\log |x_1|_v \leq \log 4 + \max(q \log |y|_v, \log |x_2|_v) \leq \log 4 + q|\log |y|_v| + |\log |x_2|_v|.
\]
But
\[
|\log |x_2|_v| = \left|\sum_{i = 1}^{s - 1} \log |\eta_i^{b_i}|_v\right| \leq \sum_{i = 1}^{s - 1} |b_i| |\log |\eta_i|_v| \leq q \sum_{i = 1}^{s - 1} 2dh(\eta_i) \leq (2s)^{O(s)}R_Sq
\]
by our choice of the fundamental system $\{\eta_1, \ldots, \eta_{s - 1}\}$ of $S$-units. Therefore,
\[
\log |x_1|_v \leq \log 4 + q|\log |y|_v| + (2s)^{O(s)}R_Sq \leq (2s)^{O(s)}R_Sq + q|\log |y|_v|.
\]
Also, by the product formula,
\[
-\log |x_1|_v = \sum_{\substack{w \in S \\ w \neq v}} \log |x_1|_w \leq (2s)^{O(s)}R_Sq + q\sum_{\substack{w \in S \\ w \neq v}} |\log |y|_w| \leq (2s)^{O(s)}R_Sq + 2dqh(y).
\]
But then
\[
|a_1 \log |\eta_1|_v + \cdots + a_{s - 1} \log |\eta_{s - 1}|_v| = |\log |x_1|_v| \leq (2s)^{O(s)}R_Sq + 2dqh(y),
\]
for all $v \in S$. Then in view of Lemma \ref{lSunit} (iii), we obtain a system of linear inequalities whose coefficient matrix has an inverse of which the elements have absolute values at most $(2s)^{O(s)}$. Consequently,
\[
W = \max(|a_1|, \ldots, |a_{s - 1}|) \leq s(2s)^{O(s)}((2s)^{O(s)}R_Sq + 2dqh(y)) \leq (2s)^{O(s)}R_Sqh(y)
\]
by (\ref{sregl}) and Lemma \ref{lVoutier}.
\end{proof}

\begin{lemma}
Assume (\ref{assq}) and $q > \cst{53}$. Then
\begin{align}
\label{lemma2}
h(y) \leq (2s)^{O(s)} R_S =: \cst{55}.
\end{align}
\end{lemma}

\begin{proof}
Fix $v \in S$. By (\ref{suniteq})
\[
|x_2|_v = |y^q - x_1|_v = |y^q|_v |1 - x_1y^{-q}|_v = |y^q|_v |1 - \zeta_1 \eta_1^{a_1} \cdots \eta_{s - 1}^{a_{s - 1}}y^{-q}|_v.
\]
We distinguish the cases that $v$ is archimedean and that $v$ is non-archimedean. First suppose that $v$ is archimedean. We apply Lemma \ref{lMatveev} with $n = s + 1$, $(\alpha_1, \ldots, \alpha_n) = (\zeta_1, \eta_1, \ldots, \eta_{s - 1}, y)$ and $(b_1, \ldots, b_n) = (1, a_1, \ldots, a_{s - 1}, -q)$.
For $i = 2, \ldots, s$, we use
\[
dh(\alpha_i) + \pi \leq d(1 + \pi \cst{1}^{-1}) h(\alpha_i).
\]
So we can take $A_i' = dh(\alpha_i) + \pi$ for $i \not \in \{2, \ldots, s\}$ and $A_i' = d(1 + \pi \cst{1}^{-1}) h(\alpha_i)$ for $i \in \{2, \ldots, s\}$. Because we need to prove that $h(y)$ is bounded, we may suppose that $h(y) > \pi$. Then it follows that $h(y) < A_n' < (d + 1)h(y)$ and $B < (2s)^{O(s)} R_S^2 q$. Lemma \ref{lMatveev} gives after enlarging $\Cst{52}$ if necessary
\[
|1 - x_1y^{-q}|_v > \exp (-\cst{56}h(y)\log q)
\]
with $\cst{56} := (2s)^{O(s)} R_S$.

Next suppose that $v$ is non-archimedean. Suppose that $v$ corresponds to a prime ideal $\mathfrak{p}$. We may assume that $\text{ord}_\mathfrak{p}(q) = 0$ since $q > \cst{53}$. We apply Lemma \ref{lYu} with $n = s + 1$, $(\alpha_1, \ldots, \alpha_n) = (\zeta_1, \eta_1, \ldots, \eta_{s - 1}, y)$ and $(b_1, \ldots, b_n) = (1, a_1, \ldots, a_{s - 1}, -q)$. Take $\delta = \frac{1}{2}$. Then $B \leq \cst{54}qh(y)$, $B_n = q$, $h_n' = h(y)$ by assuming $h(y) \geq \frac{1}{16}e^2 d^2$ and 
\[
M \leq \exp(O(s^2)) P^{d(s + 2)} R_S q.
\]
This gives
\begin{align*}
|1-x_1y^{-q}|_v &= \exp (- \log N(\mathfrak{p}) \text{ord}_{\mathfrak{p}}(1-\zeta_1 \eta_1^{a_1} \cdots \eta_{s - 1}^{a_{s - 1}}y^{-q})) \\
&> \exp\left(- (2s)^{O(s)} P \max \left((2s)^{O(s)} R_S h(y)\log(2M), \cst{54}h(y)\right)\right).
\end{align*}
By taking $\Cst{52}$ sufficiently large again, we find thanks to our assumption $q > \cst{53}$
\[
q > \sqrt[ds]{O(1)^{s^2} P^{d(s + 2)} R_S}
\]
and therefore
\[
|1-x_1y^{-q}|_v > \exp\left(- (2s)^{O(s)} P R_S h(y)\log q\right).
\]
We conclude in both cases that
\[
|1-x_1y^{-q}|_v > \exp(-\cst{91}h(y)\log q)
\]
where $\cst{91} := (2s)^{O(s)} P R_S$. Define $S_1 = \{v \in S : |y|_v > 1\}$. Then by
\[
\prod_{v \in S_1} |y|_v= \prod_{v \in M_K}\max (1, |y|_v)=\exp(d h(y))
\]
it follows
\begin{align*}
\exp(s (2s)^{O(s)}R_Sq) &\geq \prod_{v \in S_1} |x_2|_v = \exp (qdh(y))\prod_{v \in S_1} |1-x_1y^{-q}|_v \\
&> \exp(qdh(y) - s\cst{91}h(y)\log q).
\end{align*}
Making $\Cst{52}$ sufficiently large gives
\[
\sqrt{q} > \frac{2s\cst{91}}{d}.
\]
But we have the well-known inequality
\[
\frac{q}{\log q} > \sqrt{q},
\]
so
\[
qdh(y) > 2s\cst{91}h(y)\log q.
\]
We conclude that
\[
\exp(s (2s)^{O(s)}R_Sq) > \exp\left(\frac{1}{2}qdh(y)\right),
\]
hence
\[
h(y) \leq \frac{2s}{d} (2s)^{O(s)}R_S \leq (2s)^{O(s)} R_S.
\]
So we can take
\[
\cst{55} = \max\left(\pi, \frac{1}{16}e^2d^2, (2s)^{O(s)} R_S\right) \leq (2s)^{O(s)} R_S
\]
proving (\ref{lemma2}).
\end{proof}

\begin{proof}[Proof of Theorem \ref{t8}]
We showed earlier that
\[
|1-x_1y^{-q}|_v > \exp(-\cst{91}h(y)\log q)
\]
for all $v \in S$. We may assume that $q > \cst{53}$ with $\cst{53}$ sufficiently large so that (\ref{lemma2}) is valid. Then, because $x_2 = y^q(1 - x_1y^{-q})$ is an $S$-unit, we have
\begin{align*}
1 &= \prod_{v \in S} |x_2|_v \\
&= \prod_{v \in S} |y|_v^q \prod_{v \in S} |1 - x_1y^{-q}|_v \\
&\geq N_S(y)^q \exp (-s \cst{55} \cst{91} \log q),
\end{align*}
where
\[
N_S(y) = \prod_{v \in S} |y|_v.
\]
Because $y$ is a non-zero non-unit in $\mathcal{O}_S$, we have $|N_S(y)| \geq 2$. Hence
\[
1 \geq 2^q \exp(-s \cst{55} \cst{91} \log q)
\]
giving
\[
s \cst{55} \cst{91} \sqrt{q} \geq s \cst{55} \cst{91} \log q \geq q \log 2.
\]
We conclude that
\[
q \leq \left(\frac{s \cst{55} \cst{91}}{\log 2}\right)^2 \leq (2s)^{O(s)} P^2 R_S^4.
\]
This gives the desired bound for $q$, completing the proof.
\end{proof}

\subsection{Proof of Theorem \ref{tCatalana}}
We now prove Theorem \ref{tCatalana} in several steps. \\

\noindent \textbf{A: simplifications} \\
Let $x$, $y$, $p$, $q$ be a solution of (\ref{catalan}) satisfying the conditions of Theorem \ref{tCatalana}. We follow \citep{BGT} with the necessary modifications. We first show that we can make certain assumptions without loss of generality.

Note that (\ref{pqbound}) is an easy consequence of (\ref{heightbound}) and Lemma \ref{lVoutier}. So from now on we may assume that $p$ and $q$ are prime and our goal will be to show (\ref{pqprimebound}). If we have (\ref{pqprimebound}), then (\ref{heightbound}) follows from Lemma \ref{lsuper} and \ref{lhyper}. We may further assume that $p > 2$ and $q > 2$. Indeed, if e.g. $p = 2$, then we apply Lemma \ref{lST} with $f(X) = \pm (X^2 - 1)$ to conclude that $q$ is bounded.

If $q$ is a prime with $q > 2$, then $q$ is odd. Hence we may restrict our attention to the equation
\begin{align}
\label{catalan2}
x^p + y^q = 1
\end{align}
in $x, y \in \mathcal{O}_S$, $p, q \in \mathbb{N}$ with $p$ and $q$ primes, since we can replace $y$ by $-y$ when necessary.

It is further no restriction to assume that neither $x$ nor $y$ is an $S$-unit. Indeed, if both $x$ and $y$ are $S$-units, then (\ref{catalan2}) and Lemma \ref{lSuniteq} with $\alpha = \beta = 1$ imply 
\[
h(x^p) = ph(x) \leq (2s)^{O(s)} P^2 R_S^2
\]
and 
\[
h(y^q) = qh(y) \leq (2s)^{O(s)} P^2 R_S^2,
\]
whence we are done by Lemma \ref{lVoutier}. If exactly one of $x$, $y$ is an $S$-unit, $x$ say, then by applying Theorem \ref{t8} with $x_1 = -x^p$, $x_2 = 1$ to $-x^p + 1 = y^q$, we obtain 
\[
q \leq (2s)^{O(s)} P^2 R_S^4
\]
and
\[
p \leq (2s)^{O(s)} P^2 R_S^6,
\]
giving us the desired bounds.

We may also assume that $h(x) > 3$ and $h(y) > 3$. Indeed, suppose e.g. that $h(y) \leq 3$. By Theorem \ref{tNorthcott} there are only finitely many $y \in K$ such that $h(y) \leq 3$. Now take $S'$ large enough such that all $y \in K$ with $h(y) \leq 3$ become $S'$-units. If $x$ becomes an $S'$-unit, we apply Lemma \ref{lSuniteq}. Otherwise we apply Theorem \ref{t8}.

If $p = q$, then $x^p$, $-xy$ is a solution of the equation
\[
u(u - 1) = v^p
\]
in $u, v \in \mathcal{O}_S$. But $xy$ is not an $S$-unit so certainly not a root of unity. Hence, by Lemma \ref{lST}, we have 
\[
p = q \leq (2s)^{O(s)} |D_K|^{12} P^4.
\]
So we may assume without loss of generality that $p > q$. 

Finally, we may assume that $q > \cst{10} := P \geq 2$. Indeed, if $q \leq \cst{10}$, then we apply Lemma \ref{lST} with $f(Y) = 1 - Y^q$ to conclude that
\begin{align}
\label{pfinal3}
p \leq (P^2s)^{O(Ps)} |D_K|^{6P} P^{P^2}.
\end{align}

\noindent \textbf{B: a special case} \\
By A) we may restrict our attention to equation (\ref{catalan2}) in non-zero non-$S$-units $x, y \in \mathcal{O}_S$ with $h(x) > 3$ and $h(y) > 3$ and primes $p, q$ with $p > q > \cst{10} \geq 2$. We first deal with the special case that
\begin{align}
\label{ass}
(x - 1)^p + (y - 1)^q = 0,
\end{align}
which can be dealt with in an elementary way.

If $\mathfrak{p} \mid x - 1$ for some prime ideal $\mathfrak{p}$ in $\mathcal{O}_S$, then (\ref{ass}) implies $\mathfrak{p} \mid y - 1$. But it follows then from (\ref{catalan2}) that $\mathfrak{p} \mid x$. Hence $\mathfrak{p} \mid 1$ which is impossible. Thus $x - 1$ is an $S$-unit and, by (\ref{ass}), $y - 1$ is also an $S$-unit.

Subsequently we show that there is an $S$-unit $\varepsilon$ such that
\[
x = 1 - \varepsilon^q \text{ and } y = 1 + \varepsilon^p.
\]
Let $w \in \overline{\mathbb{Q}}$ be such that $w^q = 1 - x$. Then $w^{pq} = (y - 1)^q$. Hence $w^p = \rho(y - 1)$ with $\rho$ a $q$th root of unity. For any $q$th root of unity $\zeta$ we have $(\zeta w)^q = 1 - x$ and $(\zeta w)^p = \zeta^p \rho (y - 1)$. By $\gcd(p, q) = 1$ we can choose $\zeta$ such that $\zeta^p = \rho^{-1}$. Put $\varepsilon = \zeta w$. Then $\varepsilon^q = 1 - x$ and $\varepsilon^p = y - 1$. Hence $\varepsilon^p, \varepsilon^q \in K$. Since $\gcd(p, q) = 1$, we find $\varepsilon \in K$ by applying Euclid's algorithm to the exponents. But $\varepsilon^p$ is an $S$-unit, thus $\varepsilon$ is also an $S$-unit. Furthermore, 
\[
3 < h(y) \leq h(1) + h(\varepsilon^p) + \log 2 = ph(\varepsilon) + \log 2
\] 
hence $\varepsilon$ is not a root of unity. Therefore we have by Lemma \ref{lVoutier}
\begin{align}
\label{elow}
dh(\varepsilon) > \cst{1}.
\end{align}

Let $\mathfrak{p}$ be an arbitrary prime ideal divisor of $q$ in $\mathcal{O}_S$. (\ref{catalan2}) and (\ref{ass}) imply that
\begin{align}
\label{mod}
(x - 1)^p \equiv 1 - y^q \equiv x^p \mod \mathfrak{p}.
\end{align}
Since $x - 1$ is an $S$-unit, we have $\mathfrak{p} \nmid x - 1$ and so, by (\ref{mod}), $\mathfrak{p} \nmid x$. There is an $x' \in \mathcal{O}_S$ with $\mathfrak{p} \nmid x'$ and $xx' \equiv 1 \mod \mathfrak{p}$. Hence (\ref{mod}) gives
\[
((x - 1) x')^p \equiv 1 \mod \mathfrak{p}.
\]
Here $(x - 1)x' \equiv 1 - x' \not \equiv 0$ and $\not \equiv 1 \mod \mathfrak{p}$. This means that $p$ is the smallest positive integer $t$ for which
\[
(1 - x')^t \equiv 1 \mod \mathfrak{p}.
\]
But
\[
(1 - x')^{N(\mathfrak{p}) - 1} \equiv 1 \mod \mathfrak{p},
\]
hence $p \mid N(\mathfrak{p}) - 1$ in $\Z$. Since $N(\mathfrak{p}) = q^f$ with some positive integer $f \leq d$, we obtain
\begin{align}
\label{plow}
p \leq q^d.
\end{align}
Using (\ref{elow}) and (\ref{plow}), we shall now prove that $q$ is bounded. Take a place $v \in S$ such that $|\varepsilon|_v \geq H(\varepsilon)^{d/s}$. Then
\begin{align}
\label{elow2}
|\varepsilon|_v \geq H(\varepsilon)^{d/s} = \exp(h(\varepsilon)d/s) \geq 1 + h(\varepsilon)d/s > 1 + \cst{1}/s
\end{align}
by (\ref{elow}). Put
\[
f(z) = (1 - z^q)^p + (1 + z^p)^q - 1.
\]
Then
\begin{align}
\label{fzero}
0 = f(\varepsilon) = \sum_{k = 0}^p \binom{p}{k} (-\varepsilon^q)^k + \sum_{l = 0}^q \binom{q}{l} \varepsilon^{pl} - 1.
\end{align}
The leading term of $f$ is $pz^{(p - 1)q}$. First suppose that $v$ is infinite and let $\sigma: K \rightarrow \mathbb{C}$ be an embedding corresponding to $v$. We may suppose that $\sigma$ is the identity. Then $|\varepsilon|_v = |\varepsilon|^{s(v)}$ with $s(v) = 1$ if $v$ is real and $s(v) = 2$ if $v$ is complex, hence by (\ref{elow2})
\begin{align}
\label{elow3}
|\varepsilon| > \sqrt{1 + \cst{1}/s} =: 1 + \cst{4}.
\end{align}
So by (\ref{fzero}), we have
\begin{align*}
p|\varepsilon|^{(p - 1)q} &= \left|\sum_{k = 0}^{p - 2} \binom{p}{k} (-\varepsilon^q)^k + \sum_{l = 0}^{q - 1} \binom{q}{l} \varepsilon^{pl} - 1\right| \\
&\leq q|\varepsilon|^{p(q - 1)} + \sum_{k = 0}^{p - 2} \binom{p}{k} |\varepsilon|^{kq} + \sum_{l = 1}^{q - 2} \binom{q}{l} |\varepsilon|^{lp}.
\end{align*}
Combined with $p > q$ and (\ref{elow3}) this gives
\begin{align}
\label{comb}
1 &\leq |\varepsilon|^{q - p} + \frac{1}{p} \sum_{k = 0}^{p - 2} \binom{p}{k} |\varepsilon|^{(k - p + 1)q} + \frac{1}{p} \sum_{l = 1}^{q - 2} \binom{q}{l} |\varepsilon|^{(l - q)p + q} \nonumber \\
&\leq \frac{1}{|\varepsilon|} + \frac{1}{p} \sum_{k = 1}^{p - 1} \binom{p}{k + 1} |\varepsilon|^{-kq} + \frac{1}{p} \sum_{l = 1}^{q - 2} \binom{q}{l + 1} |\varepsilon|^{-lp} \nonumber \\
&< \frac{1}{|\varepsilon|} + \sum_{k = 1}^\infty p^k |\varepsilon|^{-kq} + \sum_{l = 1}^\infty q^l |\varepsilon|^{-lp},
\end{align}
and subsequently, by (\ref{plow}) and (\ref{elow3}),
\begin{align}
\label{comb2}
\frac{p}{|\varepsilon|^p} \leq \frac{p}{|\varepsilon|^q} \leq \frac{q^d}{(1 + \cst{4})^q} < \frac{\cst{4}}{4(1 + \cst{4})} < \frac{1}{2}
\end{align}
after taking $q$ sufficiently large. To find a suitable lower bound for $q$, we want to apply Lemma \ref{lEstimates2} with $x = q$, $a = d$, $b = 1 + \cst{4}$ and $c = \frac{4(1 + \cst{4})}{\cst{4}}$. So we need to check that
\[
\frac{2d}{\log(1 + \cst{1}/s)} c^{1/d} = \frac{d}{\log(1 + \cst{4})} c^{1/d} > e.
\]
Observe that $\cst{1}/s < 1$, hence $\cst{4} < 1$. This gives
\[
2d c^{1/d} \geq 4,
\]
so we can apply Lemma \ref{lEstimates2}. Lemma \ref{lEstimates2} tells us that (\ref{comb2}) holds if
\begin{align}
\label{qpoly}
q > (2ds)^{O(1)}.
\end{align}
If $q \leq (2ds)^{O(1)}$, then (\ref{plow}) gives us the desired bound for $p$. So from now on we may assume (\ref{qpoly}) and hence (\ref{comb2}).

It follows from (\ref{elow3}), (\ref{comb}) and (\ref{comb2}) that
\[
\frac{\cst{4}}{1 + \cst{4}} \leq 1 - \frac{1}{|\varepsilon|} < \sum_{k = 1}^\infty p^k |\varepsilon|^{-kq} + \sum_{l = 1}^\infty q^l |\varepsilon|^{-lp} \leq \frac{2p}{|\varepsilon|^q} + \frac{2q}{|\varepsilon|^p} < \frac{\cst{4}}{1 + \cst{4}},
\]
a contradiction.

Now suppose that $v$ is finite. Then (\ref{fzero}) implies
\begin{align*}
|p|_v |\varepsilon|_v^{(p - 1)q} &= \left|\sum_{k = 0}^{p - 2} \binom{p}{k} (-\varepsilon^q)^k + \sum_{l = 0}^{q - 1} \binom{q}{l} \varepsilon^{lp} - 1\right|_v \\
&= \left|q\varepsilon^{p(q - 1)} + \sum_{k = 0}^{p - 2} \binom{p}{k} (-\varepsilon^q)^k + \sum_{l = 1}^{q - 2} \binom{q}{l} \varepsilon^{lp}\right|_v \\
&\leq \max_{i, j} \left(|q|_v |\varepsilon|_v^{p(q - 1)}, \left| \binom{p}{i} (-\varepsilon^q)^i \right|_v, \left|\binom{q}{j} \varepsilon^{jp}\right|_v \right),
\end{align*}
where the maximum is taken over $i = 0, \ldots, p - 2$ and $j = 1, \ldots, q - 2$. Hence
\[
1 \leq \max_{i, j} \left(\left|\frac{q}{p}\right|_v |\varepsilon|_v^{q - p}, \left|\frac{1}{p} \binom{p}{i}\right|_v |\varepsilon|_v^{(i - p + 1)q}, \left|\frac{1}{p} \binom{q}{j}\right|_v |\varepsilon|_v^{(j - q)p + q}\right).
\]
If $p$ is sufficiently large as we may assume, we have
\[
\left|\frac{1}{p}\right|_v = 1.
\]
So we get by $p > q$
\[
1 \leq |\varepsilon|_v^{-1},
\]
a contradiction. \\

\noindent \textbf{C: ideal arithmetic} \\
In view of A) and B) we restrict our further attention to equation (\ref{catalan2}) in non-zero non-$S$-units $x, y \in \mathcal{O}_S$ with $h(x) > 3$ and $h(y) > 3$ and primes $p, q$ with $p > q > \cst{10} \geq 2$ such that
\begin{align}
\label{ass2}
(x - 1)^p + (y - 1)^q \neq 0.
\end{align}
For any $\alpha \in K$ we denote by $[\alpha]$ the fractional principal ideal of $\mathcal{O}_S$ generated by $\alpha$. We have, by (\ref{catalan2}),
\[
[y]^q = [1 - x][1 + x + \cdots + x^{p - 1}] = [x - 1][\beta(x - 1) + p]
\]
for some $\beta \in \mathcal{O}_S$. Assuming $p > P$, we can write
\[
[p] = \mathfrak{p}_1^{a_1} \cdots \mathfrak{p}_r^{a_r}
\]
where $\mathfrak{p}_1, \ldots, \mathfrak{p}_r$ are distinct prime ideals in $\mathcal{O}_S$, $r \leq d$, and $a_1, \ldots, a_r$ are positive integers not exceeding $d$. If, for some prime ideal $\mathfrak{p}$ and positive integer $a$, $\mathfrak{p}^a$ is a common divisor of $[x - 1]$ and $[\beta(x - 1) + p]$ then $\mathfrak{p}^a \mid [p]$ and therefore $a \leq d$. Hence we can write
\[
[x - 1] = \mathfrak{p}_1^{b_1} \cdots \mathfrak{p}_r^{b_r} \mathfrak{a}^q
\]
where $\mathfrak{a}$ is an integral ideal and $b_1, \ldots, b_r$ are rational integers with absolute values at most $d$. Since $N(\mathfrak{p}_i) = p^{f_i}$ for some positive integer $f_i \leq d$, we have
\[
p^{-d^2} \leq N(\mathfrak{p}_i^{b_i}) \leq p^{d^2} \quad (i = 1, \ldots, r).
\]
Let $h$ denote the class number of $K$. We have
\begin{align}
\label{princ}
[x - 1]^h = (\mathfrak{p}_1^{b_1} \cdots \mathfrak{p}_r^{b_r})^h \mathfrak{a}^{hq}.
\end{align}
Here $\mathfrak{a}^h = [\kappa]$ and $(\mathfrak{p}_1^{b_1} \cdots \mathfrak{p}_r^{b_r})^h = [\pi_0]$ for some $\kappa \in \mathcal{O}_S$ and $\pi_0 \in K$ such that $\pi_0 = \frac{\pi_1}{\pi_2}$ with $\pi_1, \pi_2 \in \mathcal{O}_S$ and
\begin{align}
\label{Np}
|\log N(\pi_k)| \leq d^3h \log p \quad (k = 0, 1, 2).
\end{align}
It follows from (\ref{princ}) that
\begin{align}
\label{class}
(x - 1)^h = \varepsilon \pi_0 \kappa^q
\end{align}
for some $S$-unit $\varepsilon$. By virtue of Lemma \ref{lSunit} and Lemma \ref{lhsmall} and (\ref{Np}) and (\ref{class}) there are fundamental $S$-units $\eta_1, \ldots, \eta_{s - 1}$ such that $h(\eta_i) \leq (2s)^{O(s)}$ and that
\begin{align}
\label{class2}
(x - 1)^h = \eta_1^{u_1} \cdots \eta_{s - 1}^{u_{s - 1}} \theta_0 w^q
\end{align}
where the $u_i$ are rational integers with $0 \leq u_i < q$ for $i = 1, \ldots, s - 1$, $0 \neq w \in \mathcal{O}_S$ and $0 \neq \theta_0 \in K$ with $\theta_0 = \frac{\theta_1}{\theta_2}$ such that $\theta_1, \theta_2 \in \mathcal{O}_S$ and
\begin{align}
\label{hthu}
h(\theta_k) \leq \frac{1}{d} \log N(\pi_k) + \cst{2}R + \frac{h}{d} \log Q \leq d^2h \log p + \cst{2}R + \frac{h}{d} \log Q \leq (2s)^{O(s)} R h P \log p
\end{align}
for $k = 1, 2$. By making the constant inside $O(\cdot)$ sufficiently large, (\ref{hthu}) also holds for $k = 0$. Similarly, we can write
\begin{align}
\label{classy}
(1 - y)^h = \eta_1^{v_1} \cdots \eta_{s - 1}^{v_{s - 1}} \tau_0 \sigma^p
\end{align}
with rational integers $v_i$ such that $0 \leq v_i < p$ for $i = 1, \ldots, s - 1$, and with $0 \neq \sigma \in \mathcal{O}_S$, $0 \neq \tau_0 \in K$ such that
\begin{align}
\label{htau}
h(\tau_0) \leq (2s)^{O(s)} R h P \log q.
\end{align}

\noindent \textbf{D: first bounds for $p$ and $q$} \\
Put $X = H(x)$ and $Y = H(y)$. In this section our goal will be to show that
\begin{align}
\label{pl3}
p \leq O(1) d^{13} s P \log Y \log p.
\end{align}
Let $v \in S$ be such that $|x|_v \geq H(x)^{d/s}$. It follows from (\ref{catalan2}) that
\begin{align}
\label{lambda1}
\Lambda_1 := 1 - \frac{(-y)^q}{x^p} = \frac{1}{x^p},
\end{align}
whence
\begin{align}
\label{lambda1u}
|\Lambda_1|_v = \frac{1}{|x|_v^p} \leq X^{-pd/s}.
\end{align}
If $v$ is infinite, embed $K$ in $\mathbb{C}$ using an embedding $\sigma$ corresponding to $v$. We use  Lemma \ref{lMatveev} with $n = 2$, $(\alpha_1, \alpha_2) = (-y, x)$ and $(b_1, b_2) = (q, -p)$, giving
\begin{align}
\label{lambda1l}
\left|1 - \frac{(-y)^q}{x^p}\right|_v > e^{-O(1) d^5 \log X \log Y \log(3ep)}.
\end{align}
Assuming $p > 3e$, (\ref{lambda1u}) and (\ref{lambda1l}) imply
\begin{align}
\label{pl}
p \leq O(1) d^4 s \log Y \log p,
\end{align}
hence (\ref{pl3}).

If $v$ is finite, we apply Lemma \ref{lYu} with $n = 2$, $(\alpha_1, \alpha_2) = (-y, x)$ and $(b_1, b_2) = (q, p)$. So we can take $B = B_n = p$ and $\delta = \frac{1}{2}$. Recall that $p > q > P$, hence the conditions are satisfied. Because we want to prove (\ref{pl3}) in the case $v$ finite, we may assume that
\[
p > d P \log Y.
\]
Hence
\begin{align}
\label{lambda1l2}
|\Lambda_1|_v > \exp\left(-O(1) d^{14} P \log X \log Y \log p \right).
\end{align}
Now (\ref{lambda1u}) and (\ref{lambda1l2}) imply
\begin{align}
\label{pl2}
p \leq O(1) d^{13} s P \log Y \log p.
\end{align}
So in all cases we have (\ref{pl3}). By estimating $|\Lambda_2|_v$ with $\Lambda_2 := 1 - \frac{(-x)^p}{y^q} = \frac{1}{y^q}$ we can prove in a similar way that
\begin{align}
\label{ql}
q \leq O(1) d^{13} s P \log X \log p.
\end{align}

\noindent \textbf{E: a bound for $q$} \\
We shall now prove that
\begin{align}
\label{qlogl}
q < \cst{5} (\log p)^4
\end{align}
with $\cst{5} = (2s)^{O(s)} R^3 h^3 P^4 R_S^2$. To prove this we may assume that
\begin{align}
\label{ass3}
q > \log p.
\end{align}
Further, we may assume that
\begin{align}
\label{XYl}
\min(X, Y) > p^{\cst{6}}
\end{align}
with $\cst{6} := 4s/d$. Indeed if $Y \leq p^{\cst{6}}$ then $q < p \leq O(1) d^{13} s P \cst{6} (\log p)^2$ follows from (\ref{pl3}), implying (\ref{qlogl}). Further, in case $X \leq p^{\cst{6}}$, (\ref{qlogl}) immediately follows from (\ref{ql}). Let $v \in S$ be such that $|x|_v \geq X^{d/s}$. From (\ref{catalan2}) we obtain
\begin{align}
\label{bound1}
\left|\frac{(-y)^q}{x^p} - 1\right|_v = \frac{1}{|x|_v^p}.
\end{align}
We combine the cases $v$ real, $v$ complex and $v$ finite. Note that in all cases $|\cdot|_v^{1/2}$ satisfies the triangle inequality. Because $\cst{6} = 4s/d$, we get $|x|_v \geq 12$. Hence
\[
|x - 1|_v^{1/2} \geq |x|_v^{1/2} - |1|_v^{1/2} = |x|_v^{1/2} - 1 \geq \frac{1}{2} \sqrt{2} |x|_v^{1/2}
\]
and 
\begin{align}
\label{xminus1}
|x - 1|_v \geq \frac{1}{2}|x|_v \geq p^2
\end{align}
where we have used again $\cst{6} = 4s/d$. It follows that
\[
\left|\frac{x^p}{(x - 1)^p} - 1\right|_v^{1/2} = \left|\frac{((x - 1) + 1)^p - (x - 1)^p}{(x - 1)^p}\right|_v^{1/2} \leq \sum_{i = 1}^p \left(\frac{p^i}{|x - 1|_v^i}\right)^{1/2} \leq p\left(\frac{p}{|x - 1|_v}\right)^{1/2}
\]
and after squaring
\begin{align}
\label{bound2}
\left|\frac{x^p}{(x - 1)^p} - 1\right|_v \leq \frac{p^3}{|x - 1|_v} \leq \frac{2p^3}{|x|_v}.
\end{align}
Furthermore, by (\ref{catalan2}), $p > q$ and $|x|_v \geq 12$
\[
\frac{|y|_v^{q/2}}{|x|_v^{q/2}} \geq \frac{|x|_v^{p/2} - 1}{|x|_v^{q/2}} \geq \frac{|x|_v^{p/2} - 1}{|x|_v^{p/2}} \geq \frac{1}{2} > \left(\frac{1}{2}\right)^{q/2}.
\]
We conclude that
\begin{align}
\label{ymax}
|y|_v \geq \frac{1}{2}|x|_v \geq p^2 > q.
\end{align}
Hence we have
\[
\left|\frac{(1 - y)^q}{(-y)^q} - 1\right|_v^{1/2} = \left|\frac{(1 - y)^q + y^q}{(-y)^q}\right|_v^{1/2} \leq \sum_{i = 1}^q \left(\frac{q^i}{|y|_v^i}\right)^{1/2} \leq q \left(\frac{q}{|y|_v}\right)^{1/2}
\]
and after squaring
\begin{align}
\label{bound3}
\left|\frac{(1 - y)^q}{(-y)^q} - 1\right|_v \leq \frac{q^3}{|y|_v} \leq \frac{2p^3}{|x|_v}.
\end{align}
From (\ref{bound1}), (\ref{bound2}), (\ref{bound3}) and the identity
\[
z_1z_2z_3 - 1 = \prod_{i = 1}^3 (z_i - 1) + \sum_{1 \leq i < j \leq 3} (z_i - 1)(z_j - 1) + \sum_{i = 1}^3 (z_i - 1),
\]
we infer
\begin{align}
\label{bound4}
\left|\frac{(1 - y)^q}{(x - 1)^p} - 1\right|_v \leq \frac{26 p^6}{|x|_v} = \frac{O(1) p^6}{|x|_v}.
\end{align}
Further we have, by (\ref{catalan2}), (\ref{xminus1}) and (\ref{ymax}),
\begin{align}
\label{bound5}
\left|\frac{(1 - y)^q}{(x - 1)^p}\right|_v^{1/2} 
&= \left|\frac{(1 - y)^q}{y^q}\right|_v^{1/2} \cdot \left|\frac{1 - x^p}{(x - 1)^p}\right|_v^{1/2}
\leq 2\left(1 + \frac{1}{|y|_v^{1/2}}\right)^q \left(1 + \frac{1}{|x - 1|_v^{1/2}}\right)^p \nonumber\\
&\leq 2\left(1 + \frac{\sqrt{2}}{|x|_v^{1/2}}\right)^{p + q} \leq 2\left(1 + \frac{2}{p}\right)^{2p} \leq 2e^4 = O(1).
\end{align}
For
\begin{align}
\label{lambda3}
\Lambda_3 := \frac{(1 - y)^{qh}}{(x - 1)^{ph}} - 1
\end{align}
we obtain, from (\ref{bound4}) and (\ref{bound5}),
\begin{align}
\label{lambda3u}
|\Lambda_3|_v < \frac{O(1) \left(1 + O(1) + \cdots + O(1)^{h - 1}\right)^2 p^6}{|x|_v} \leq \frac{O(1)^{2h} p^6}{|x|_v}.
\end{align}

Suppose now that $\Lambda_3 \neq 0$, i.e. that $(x - 1)^{ph} \neq (1 - y)^{qh}$. Using (\ref{lambda3}), (\ref{class2}) and (\ref{classy}), we obtain
\[
\Lambda_3 = \eta_1^{e_1} \cdots \eta_{s - 1}^{e_{s - 1}} \tau_0^q \theta_0^{-p} \left(\frac{\sigma}{w}\right)^{pq} - 1
\]
where $e_i \in \Z$ with $|e_i| \leq pq$ for $i = 1, \ldots, s - 1$. Put $H_1 = H(\sigma)$, $H_2 = H(w)$ and $H_0 = \max(H_1, H_2)$. Then
\begin{align}
\label{H0}
H\left(\frac{\sigma}{w}\right) \leq H(\sigma)H(w) \leq H_0^2.
\end{align}
First suppose that $v$ is infinite. By applying Lemma \ref{lMatveev} to $\Lambda_3$ and using (\ref{hthu}), (\ref{htau}), (\ref{H0}) and $p > q$ we obtain
\[
|\Lambda_3|_v > \exp(-(2s)^{O(s)} R^2 h^2 P^2 R_S (\log p)^3 \log^\ast H_0).
\]
Next suppose that $v$ is finite. By applying Lemma \ref{lYu} to $\Lambda_3$ and using (\ref{hthu}), (\ref{htau}), (\ref{H0}) and $p > q$ we obtain
\[
|\Lambda_3|_v > \exp(-(2s)^{O(s)} R^2 h^2 P^3 R_S (\log p)^3 \log^\ast H_0)
\]
if $pq > s R h P R_S \log p \log q$. This together with (\ref{lambda3u}) gives in all cases
\begin{align}
\label{XH0}
d/s \log X \leq \log |x|_v \leq (2s)^{O(s)} R^2 h^2 P^3 R_S (\log p)^3 \log^\ast H_0.
\end{align}
If $H_0 \leq \cst{7} := e^{(2s)^{O(s)} R h P R_S}$, then (\ref{ql}) and (\ref{XH0}) give (\ref{qlogl}). We therefore assume that $H_0 > \cst{7}$.

First suppose that $H_2 > \cst{7}$. Then, by (\ref{hthu}) and (\ref{ass3}), we have
\[
\left|\frac{1}{\theta_0}\right|_v \leq H\left(\frac{1}{\theta_0}\right) = H(\theta_0) \leq e^{(2s)^{O(s)} R h P \log p} < e^{(2s)^{O(s)} R h P q} \leq H_2^{\frac{q}{4s}}
\]
for all $v \in S$ by taking the constant inside $O(\cdot)$ sufficiently large. Hence we obtain from (\ref{class2})
\[
|w|_v^q \leq |x - 1|_v^h \left|\frac{1}{\theta_0}\right|_v \cdot \prod_{i = 1}^{s - 1} \left|\frac{1}{\eta_i}\right|_v^{u_i} \leq |x - 1|_v^h H_2^{\frac{q}{4s}} e^{(s - 1) (2s)^{O(s)} R_S q} < 4^h X^{dh} H_2^{\frac{q}{3s}}
\]
again by taking the constant inside $O(\cdot)$ sufficiently large. Choosing $v \in S$ such that $|w|_v \geq H_2^{d/s}$, we obtain
\[
4^h X^{dh} H_2^{\frac{q}{3s}} > |w|_v^q \geq H_2^{qd/s}.
\]
Consequently, we have
\begin{align}
\label{XH02}
h \log X^d > \frac{qd}{s} \log H_2 - \log(4^hH_2^{\frac{q}{3s}}) \geq \left(\frac{d}{s} - \frac{2}{3s} \right) q \log H_2 \geq \frac{d}{3s} q \log H_2
\end{align}
if $H_2^{\frac{q}{3s}} \geq \cst{7}^{\frac{q}{3s}} \geq 4^h$. By using (\ref{classy}) and (\ref{htau}) one can prove in a similar manner that
\begin{align}
\label{YH0}
\log Y > \frac{1}{3hs} p \log H_1
\end{align}
if $H_1 > \cst{7}$. If $H_0 = H_2$, then (\ref{XH0}) and (\ref{XH02}) imply
\[
q < (2s)^{O(s)} R^2 h^3 P^3 R_S (\log p)^3,
\]
hence (\ref{qlogl}). Next suppose $H_0 = H_1$. From (\ref{catalan2}) we obtain
\[
qh(y) = h(y^q) = h(x^p - 1) \leq \log 2 + h(x^p) + h(1) = \log 2 + ph(x),
\]
so
\begin{align}
\label{XY}
q \log Y < \left(1 + \frac{d}{\cst{1}} \log 2 \right) p \log X.
\end{align}
Now (\ref{XH0}), (\ref{YH0}) and (\ref{XY}) imply
\[
\frac{1}{3hs} pq \log H_0 < q \log Y < \left(1 + \frac{d}{\cst{1}} \log 2 \right) p \log X <  (2s)^{O(s)} R^2 h^2 P^3 R_S p(\log p)^3 \log^\ast H_0,
\]
whence (\ref{qlogl}). \\

\noindent \textbf{F: completing the proof of E)} \\
To prove (\ref{qlogl}) we are left with the case
\begin{align}
\label{ass4}
(x - 1)^{ph} = (1 - y)^{qh}.
\end{align}
We can now repeat the argument of part E) above with
\[
\Lambda_4 := \frac{(1 - y)^q}{(x - 1)^p} - 1
\]
instead of $\Lambda_3$. By assumption (\ref{ass2}) we have $\Lambda_4 \neq 0$. We still need to derive a lower bound for $|\Lambda_4|_v$. Note that $\frac{(1 - y)^q}{(x - 1)^p}$ is a $h$-th root of unity, hence
\[
\frac{1}{d} \log |\Lambda_4|_v \geq -h(\Lambda_4) = -h\left(\frac{(1 - y)^q}{(x - 1)^p} - 1\right) \geq -\log 2 - h\left(\frac{(1 - y)^q}{(x - 1)^p}\right) - h(-1) = - \log 2.
\]
We conclude that
\[
|\Lambda_4|_v \geq 2^{-d}.
\]
Now inequality (\ref{qlogl}) follows. \\

\noindent \textbf{G: finishing the proof} \\
We shall now prove that $p$ is bounded from above by using (\ref{pl3}) and (\ref{qlogl}). By (\ref{pl3}) we may assume that $Y > 4^{s/d}$. Let $v \in S$ be such that $|y|_v \geq Y^{d/s} \geq 4$. Then, by (\ref{catalan2}),
\begin{align}
\label{fbound1}
\left|\frac{x^p}{(1 - y)^q}\right|_v^{1/2} = \left|\frac{1 - y^q}{(1 - y)^q}\right|_v^{1/2} \leq \frac{2|y|_v^{q/2}}{(|y|_v^{1/2}/2)^q} \leq 4^q.
\end{align}
Hence, using again (\ref{catalan2}),
\begin{align}
\label{fbound2}
\left|\frac{x^p}{(1 - y)^q} - 1\right|_v^{1/2} = \left|\frac{x^p + (y - 1)^q}{(1 - y)^q}\right|_v^{1/2} \leq \frac{q 2^{q/2} |y|_v^{(q - 1)/2}}{(|y|_v^{1/2}/2)^q} \leq \frac{4^q}{|y|_v^{1/2}}.
\end{align}
Putting
\[
\Lambda_5 := \frac{x^{ph}}{(1 - y)^{qh}} - 1,
\]
it follows from (\ref{fbound1}) and (\ref{fbound2}) that
\begin{align}
\label{lambda5u}
|\Lambda_5|_v < \frac{16^q \left(1 + 4^q + \cdots + 4^{q(h - 1)}\right)^2}{|y|_v} \leq \frac{16^{q(h + 1)}}{|y|_v}.
\end{align}
Suppose that $|\Lambda_5| \neq 0$, i.e. that $x^{ph} \neq (1 - y)^{qh}$. We are going to derive a lower bound for $|\Lambda_5|$. By (\ref{classy}) we have
\[
\frac{x^{ph}}{(1 - y)^{qh}} = \eta_1^{d_1} \cdots \eta_{s - 1}^{d_{s - 1}} \tau_0^{-q} \left(\frac{x^h}{\sigma^q}\right)^p
\]
with rational integers $d_i$ such that $|d_i| < pq$ for $i = 1, \ldots, s - 1$. We claim that
\[
|x|_v \geq \frac{1}{2} H(x)^{d/s}.
\]
To prove our claim, we note that
\[
|y^q|_v = |1 - x^p|_v \leq 4 \max(1, |x^p|_v)
\]
and
\[
H(x^p) = H(1 - y^q) \leq 2H(y^q).
\]
Combining gives
\[
|x|_v^p = |x^p|_v \geq \frac{1}{4}|y^q|_v - 1 \geq \frac{1}{4}H(y)^{qd/s} - 1 \geq \frac{1}{8} H(x)^{pd/s} - 1 \geq \left(\frac{1}{2}\right)^p H(x)^{pd/s}
\]
if $p \geq 4$, proving the claim. Hence, by (\ref{fbound1}) and (\ref{htau}),
\[
\left|\frac{x^h}{\sigma^q}\right|_v \leq 16^{h/p} \left(\prod_{i = 1}^{s - 1} |\eta_i|_w^{-d_i}\right)^{1/p} |\tau_0|_w^{q/p} \leq 16^h e^{d(s - 1) (2s)^{O(s)} R_S q} q^{d(2s)^{O(s)} R h P}.
\]
So
\[
\left(\frac{1}{2}\right)^hH(x^h)^{d/s} \leq |x^h|_v \leq |\sigma^q|_v 16^h e^{d(s - 1) (2s)^{O(s)} R_S q} q^{d(2s)^{O(s)} R h P}.
\]
Put $H_3 = H(\sigma)$. Then
\begin{align}
\label{Hsig}
H\left(\frac{x^h}{\sigma^q}\right) \leq H(x^h)H(\sigma)^q \leq \left(32^h e^{d(s - 1) (2s)^{O(s)} R_S q} q^{d(2s)^{O(s)} R h P}\right)^{s/d}  H_3^{q(1 + s)}.
\end{align}
First suppose that $v$ is infinite. By applying Lemma \ref{lMatveev} to 
\[
\Lambda_5 = \eta_1^{d_1} \cdots \eta_{s - 1}^{d_{s - 1}} \tau_0^{-q} \left(\frac{x^h}{\sigma^q}\right)^p - 1
\]
and using (\ref{htau}) and (\ref{Hsig}), we obtain
\begin{align}
|\Lambda_5|_v > \exp(-(2s)^{O(s)} R^2 h^2 P^2 R_S^2 q (\log p)^2 \log^\ast H_3).
\end{align}
Next suppose that $v$ is finite. By applying Lemma \ref{lYu} to $\Lambda_5$ and using (\ref{htau}), we obtain
\begin{align}
|\Lambda_5|_v > \exp(-(2s)^{O(s)} R^2 h^2 P^3 R_S^2 q (\log p)^2 \log^\ast H_3)
\end{align}
if $pq > s R h P R_S \log p$. So in all cases
\begin{align}
\label{lambda5l}
|\Lambda_5|_v > \exp(-(2s)^{O(s)} R^2 h^2 P^3 R_S^2 q (\log p)^2 \log^\ast H_3).
\end{align}
Comparing (\ref{lambda5u}) and (\ref{lambda5l}) we obtain
\begin{align}
\label{YH3}
\log Y \leq s/d \log |y|_v \leq (2s)^{O(s)} R^2 h^2 P^3 R_S^2 q (\log p)^2 \log^\ast H_3.
\end{align}
If $H_3 \leq \cst{9} := e^{(2s)^{O(s)} R h P R_S}$ then (\ref{YH3}) together with (\ref{pl3}) and (\ref{qlogl}) yields 
\[
p \leq (2s)^{O(s)} R^6 h^6 P^9 R_S^5 (\log p)^7.
\]
Suppose now that $H_3 > \cst{9}$. Then we have, analogously to (\ref{YH0}),
\begin{align}
\label{YH32}
\log Y > \frac{1}{3hs} p \log H_3.
\end{align}
From (\ref{qlogl}), (\ref{YH3}) and (\ref{YH32}) it follows now again that 
\[
p \leq (2s)^{O(s)} R^5 h^6 P^7 R_S^4 (\log p)^6.
\]
So in all cases
\[
p \leq (2s)^{O(s)} R^6 h^6 P^9 R_S^5 (\log p)^7,
\]
whence 
\begin{align}
\label{pfinal}
p \leq (2s)^{O(s)} R^{12} h^{12} P^{18} R_S^{10}.
\end{align}
Using the well-known inequalities
\[
Rh \leq |D_K|^{1/2} (\log^\ast |D_K|)^{d - 1}
\]
and
\[
R_S \leq Rh \prod_{i = 1}^t \log N(\mathfrak{p}_i) \leq |D_K|^{1/2} (\log^\ast |D_K|)^{d - 1} (\log P)^t,
\]
we get from (\ref{pfinal})
\begin{align}
\label{pfinal2}
p \leq (2s)^{O(s)} |D_K|^{11} (\log^\ast |D_K|)^{22(d - 1)} P^{18} (\log P)^{10t},
\end{align}
completing the proof. Recall that in A) we assumed that $q > \cst{10} := P \geq 2$. If $q \leq \cst{10}$, we derived (\ref{pfinal3}). But observe that (\ref{pfinal3}) gives a significantly larger bound for $p$ than (\ref{pfinal2}). So our final bound for $p$ is (\ref{pfinal3}). \\

\noindent \textbf{H: the remaining case} \\
We are left with the case $x^{ph} = (1 - y)^{qh}$. This implies
\[
(1 - y^q)^h = (1 - y)^{qh}
\]
and hence
\begin{align}
\label{yrel}
1 - y^q = \zeta (1 - y)^q,
\end{align}
where $\zeta$ is some $h$-th root of unity. Put
\[
\Lambda_6 := -\zeta^{-1} \left(\frac{y}{1 - y}\right)^q - 1.
\]
Take a valuation $v \in S$ such that $|1 - y|_v \geq H(1 - y)^{d/s}$. We start by deriving an upper bound for $|\Lambda_6|_v$ using (\ref{yrel})
\[
|\Lambda_6|_v = \left|-\zeta^{-1}\left(\frac{y}{1 - y}\right)^q - 1\right|_v = |-\zeta^{-1} (1 - y)^{-q}|_v = |(1 - y)^{-q}|_v \leq H(1 - y)^{-qd/s}.
\]
Next we derive a lower bound for $|\Lambda_6|_v$. By extending $K$ if necessary we may assume that $\zeta$ is in $K$. This increases the degree of $K$ by at most a factor $h$. First suppose that $v$ is infinite. By applying Lemma \ref{lMatveev} to $\Lambda_6$ we obtain
\[
\log |\Lambda_6|_v > -O(1) h^4d^4 \mbox{$h(\frac{y}{1 - y})$} \log q.
\]
Next suppose that $v$ is finite. By applying Lemma \ref{lYu} we obtain
\[
\log |\Lambda_6|_v > -O(1) h^{13}d^{13} P \mbox{$h(\frac{y}{1 - y})$} \log q
\]
if $q > dP$. Recall that we have assumed $h(y) > 3$, so
\[
\mbox{$h(\frac{y}{1 - y})$} = h((1 - y)^{-1} - 1) \leq \log 2 + h((1 - y)^{-1}) = \log 2 + h(1 - y) \leq 2h(1 - y).
\]
Combining everything gives
\[
\frac{-qd}{s} h(1 - y) \geq -O(1) h^{13}d^{13}P h(1 - y) \log q
\]
in all cases, which can be rewritten as
\[
q \leq O(1) h^{13}d^{12}Ps \log q.
\]
We conclude that
\begin{align}
\label{qyrel}
q \leq O(1)h^{14}d^{13}P^2s^2.
\end{align}
Let $f$ be the minimal polynomial of $y$ and let $g(X) := (1 - X)^{qh} - (1 - X^q)^h$. By $H$ we denote the naive height of a polynomial. Using (\ref{yrel}) and (\ref{qyrel}) together with Lemma 3.11 and remark 2 on page 81 in \citep{Waldschmidt} we get
\begin{align*}
h(y) &\leq \frac{1}{\deg{f}} \log H(f) + O(1) \\
&\leq \log H(f) + O(1) \\
&\leq \log(2^{\deg{g}} \sqrt{\deg{g} + 1} H(g)) + O(1) \\ 
&\leq O(1)qh + O(1) \\
&\leq O(1)h^{15}d^{13}P^2s^2.
\end{align*}
Now it is straightforward to give an upper bound for $p$ and inequality (\ref{pfinal2}) follows.

\section{Proof of Theorem \ref{tMT}}
\label{Specialization}
In this section we will bound $p$ and $q$ for the Catalan equation over finitely generated domains. We will follow \citep{Brindza}. \\

We recall some notation. Let $A = \Z[z_1, \ldots, z_r]$ be an integral domain finitely generated over $\Z$ with $r > 0$ and denote by $K$ the quotient field of $A$. We have
\[
A \cong \Z[X_1, \ldots, X_r]/I
\]
where $I$ is the ideal of polynomials $f \in \Z[X_1, \ldots, X_r]$ such that $f(z_1, \ldots, z_r) = 0$. Then $I$ is finitely generated. Let $d \geq 1$, $h \geq 1$ and assume that
\[
I = (f_1, \ldots, f_m)
\]
with $\deg f_i \leq d$, $h(f_i) \leq h$ for $i = 1, \ldots, m$. Here $\deg$ means the total degree of the polynomial $f_i$ and $h(f_i)$ is the logarithmic height of $f_i$. Our goal is to prove Theorem \ref{tMT}.

\begin{proof}
Let $x$, $y$, $p$, $q$ be an arbitrary solution. Without loss of generality we may assume that $z_1, \ldots, z_k$ forms a transcendence basis of $K/\Q$. We write $t := r - k$ and rename $z_{k + 1}, \ldots, z_r$ as $y_1, \ldots, y_t$ respectively. Define
\[
A_0 := \Z[z_1, \ldots, z_k], K_0 := \Q(z_1, \ldots, z_k).
\]
Then
\[
A = A_0[y_1, \ldots, y_t], K = K_0(y_1, \ldots, y_t).
\]
By Corollary 3.4 in \citep{EG} we have $K = K_0(u)$, $u \in A$, $u$ is integral over $A_0$, and $u$ has minimal polynomial
\[
F(X) = X^D + F_1X^{D - 1} + \cdots + F_D
\]
over $K_0$ with $F_i \in A_0$, $\deg F_i \leq (2d)^{\exp O(r)}$ and $h(F_i) \leq (2d)^{\exp O(r)} (h + 1)$. Furthermore, Lemma 3.2(i) in \citep{EG} tells us that $D \leq d^t$.

By Lemma 3.6 in \citep{EG} there exists non-zero $f \in A_0$ such that
\[
A \subseteq B := A_0[u, f^{-1}]
\]
and moreover $\deg f \leq (2d)^{\exp O(r)}$ and $h(f) \leq (2d)^{\exp O(r)} (h + 1)$. From now on, we will work in the larger ring $B$ to bound $p$ and $q$. So we will assume that $x, y \in B$ and bound $p$ and $q$.

We distinguish two cases. First, we consider the case $k = 0$. In this case we have $A_0 = \Z$, $K_0 = \Q$ and $t = r$. Then $K$ is a number field of degree $D \leq d^t$ and
\[
|D_K| \leq |\text{disc}(F)| \leq D^{2D - 1} \exp\left((2d)^{\exp O(r)} (h + 1)\right) \leq \exp\left((2d)^{\exp O(r)} (h + 1)\right)
\]
by using the result on the bottom of page 335 in \citep{LM}. Let $S$ contain all infinite valuations and all prime ideal divisors of $f$. Write $s = |S|$. Let $\mathfrak{p}_1, \ldots, \mathfrak{p}_n$ be the prime ideals in $S$. Put
\[
P := \max\{2, N(\mathfrak{p}_1), \ldots, N(\mathfrak{p}_n)\}
\]
and
\[
Q := N(\mathfrak{p}_1 \cdots \mathfrak{p}_n).
\]
By $h(f) \leq (2d)^{\exp O(r)} (h + 1)$, it follows that
\[
s \leq (2d)^{\exp O(r)} (h + 1)
\]
and
\[
P \leq \exp\left((2d)^{\exp O(r)} (h + 1)\right).
\]
We conclude that
\[
Q \leq |f|^D \leq \exp\left((2d)^{\exp O(r)} (h + 1)\right)
\]
and we can apply Theorem \ref{tCatalana} to get (\ref{resulta}).

Now consider the case $k > 0$. Fix an algebraic closure $\overline{K_0}$ of $K_0$. Put
\[
T_i = \{z_1, \ldots, z_k\} \setminus \{z_i\}.
\]
Let $k_i$ be the algebraic closure of $\Q(T_i)$ in $\overline{K_0}$. Thus, $A_0$ is contained in $k_i[z_i]$. Define
\[
M_i := k_i(z_i, u^{(1)}, \ldots, u^{(D)}),
\]
where $u^{(1)}, \ldots, u^{(D)}$ are the conjugates of $u$ over $K_0$. We need the following lemma.

\begin{lemma} 
\label{lIntersection}
We have that
\[
\bigcap_{i = 1}^k k_i = \overline{\Q}.
\]
\end{lemma}

\begin{proof}
See \citep{Brindza}.
\end{proof}

First assume that $x \in k_i$ for all $i = 1, \ldots, k$. In this case $x$ and $y$ belong to the algebraic number field $\overline{\Q} \cap K$. Our goal will be to apply Theorem \ref{tCatalana}. For this, we will use a so called specialization argument. If we knew of a way to effectively compute $\overline{\Q} \cap K$, this would simplify our argument below. 

Recall that $K = K_0(u)$, $u \in A$, $u$ is integral over $A_0$, and $u$ has minimal polynomial
\[
F(X) = X^D + F_1X^{D - 1} + \cdots + F_D
\]
over $K_0$ with $F_i \in A_0$, $\deg F_i \leq (2d)^{\exp O(r)}$ and $h(F_i) \leq (2d)^{\exp O(r)} (h + 1)$. In the case $D = 1$, we take $u = 1$, $F(X) = X - 1$.

Let $\mathbf{y} = (y_1, \ldots, y_k) \in \Z^k$. We put
\[
|\mathbf{y}| := \max(|y_1|, \ldots, |y_k|).
\]
The substitution $z_1 \mapsto y_1, \ldots, z_k \mapsto y_k$ defines a ring homomorphism (specialization)
\[
\varphi_{\mathbf{y}}: \alpha \mapsto \alpha(\mathbf{y}): \{g_1/g_2 : g_1, g_2 \in A_0, g_2(\mathbf{y}) \neq 0\} \rightarrow \Q.
\]
Let us extend this to a ring homomorphism from $B$ to $\overline{\Q}$ for which we need to impose some restrictions on $\mathbf{y}$. Denote by $\Delta_F$ the discriminant of $F$, and let
\[
H := \Delta_F F_D f.
\]
It follows that $H \in A_0$. Using that $\Delta_F$ is a polynomial of degree $2D - 2$ with integer coefficients in $F_1, \ldots, F_D$, it follows easily that
\[
\deg H \leq (2d)^{\exp O(r)}.
\]
Let $N$ be an integer with $N \geq (2d)^{\exp O(r)}$. Lemma 5.4 in \citep{EG} implies that if $N \geq \deg H$ then
\[
T := \{\mathbf{y} \in \Z^k : |\mathbf{y}| \leq N, H(y) \neq 0\}
\]
is non-empty. Take $\mathbf{y} \in T$ and consider the polynomial
\[
F_\mathbf{y} := X^D + F_1(\mathbf{y})X^{D - 1} + \cdots + F_D(\mathbf{y}),
\]
which has $D$ distinct zeros which are all different from $0$, say $u_1(\mathbf{y}), \ldots, u_D(\mathbf{y})$. Thus, for $j = 1, \ldots, D$ the assignment
\[
z_1 \mapsto y_1, \ldots, z_k \mapsto y_k, u \mapsto u_j(\mathbf{y})
\]
defines a ring homomorphism $\varphi_{\mathbf{y}, j}$ from $B$ to $\overline{\Q}$. It is obvious that $\varphi_{\mathbf{y}, j}$ is the identity on $B \cap \Q$. Thus, if $\alpha \in B \cap \overline{\Q}$, then $\varphi_{\mathbf{y}, j}(\alpha)$ has the same minimal polynomial as $\alpha$ and so it is conjugate to $\alpha$.

Define the algebraic number fields $K_{\mathbf{y}, j} := \Q(u_j(\mathbf{y}))$ ($j = 1, \ldots, D$). Denote by $\Delta_L$ the discriminant of an algebraic number field $L$. Then for $j = 1, \ldots, D$ we have by Lemma 5.5 in \citep{EG} that $[K_{\mathbf{y}, j} : \Q] \leq D$ and
\begin{align*}
|\Delta_{K_{\mathbf{y}, j}}| \leq D^{2D - 1} \left(d_0^k \cdot e^{h_0} \cdot \max(1, |\mathbf{y}|)^{d_0}\right)^{2D - 2},
\end{align*}
where
\[
d_0 \geq \max(\deg F_1, \ldots, \deg F_D), \quad h_0 \geq \max(h(F_1), \ldots, h(F_D)).
\]
So we can take $d_0 = (2d)^{\exp O(r)}$ and $h_0 = (2d)^{\exp O(r)}(h + 1)$ giving
\begin{align*}
|\Delta_{K_{\mathbf{y}, j}}| &\leq D^{2D - 1} \left((2d)^{k \exp O(r)} \cdot \exp\left((2d)^{\exp O(r)} (h + 1)\right) \cdot (2d)^{(2d)^{\exp O(r)}}\right)^{2D - 2} \\
&\leq \exp\left((2d)^{\exp O(r)} (h + 1)\right).
\end{align*}

Now pick any $j = 1, \ldots, D$. Let $S$ contain all infinite valuations and all prime ideal divisors of $f(\mathbf{y})$. Then $\varphi_{\mathbf{y}, j}$ maps $B$ to the ring of $S$-integers of $K_{\mathbf{y}, j}$. In order to apply Theorem \ref{tCatalana}, we still need to bound $s$, $P$ and $Q$. 

It is easy to verify that for any $g \in A_0$, $\mathbf{y} \in \Z^k$,
\[
\log |g(\mathbf{y})| \leq k \log \deg g + h(g) + \deg g \log \max(1, |\mathbf{y}|).
\]
Applying this with $f$ and $\mathbf{y}$ we get
\[
|f(\mathbf{y})| \leq (2d)^{k \exp O(r)} \cdot \exp\left((2d)^{\exp O(r)} (h + 1)\right) \cdot (2d)^{(2d)^{\exp O(r)}} \leq \exp\left((2d)^{\exp O(r)} (h + 1)\right).
\]
Hence
\[
s \leq (2d)^{\exp O(r)} (h + 1)
\]
and
\[
P \leq \exp\left((2d)^{\exp O(r)} (h + 1)\right).
\]
We conclude that
\[
Q \leq |f(\mathbf{y})|^D \leq \exp\left((2d)^{\exp O(r)} (h + 1)\right)
\]
and we can apply Theorem \ref{tCatalana} to get (\ref{resulta}).

We still need to deal with the case $x \not \in k_i$ for some $i$. So pick an $i$ such that $x \not \in k_i$, then also $y \not \in k_i$. Let $S$ denote the subset of valuations $v$ of $M_i/k_i$ such that $v(z_i) < 0$, $v(f) > 0$, $v(x) > 0$ or $v(y) > 0$. Now let $v$ be any valuation such that $v \not \in S$. We claim that
\[
v(x) = v(y) = v(1) = 0.
\]
Because $v \not \in S$, it follows that $v(z_i) \geq 0$. Recall that $u$ is integral over $k[z_i]$. Together this implies that $v(u) \geq 0$. We also have that $v(f) \leq 0$, hence $v(f^{-1}) \geq 0$. But $x, y \in B$, so we get $v(x), v(y) \geq 0$. But then
\[
v(x) = v(y) = v(1) = 0
\]
as claimed.

Define $\Delta_i = [M_i : k_i(z_i)]$. Each valuation of $k_i(z_i)$ can be extended to at most $\Delta_i$ valuations of $M_i$. Hence $M_i$ has at most $\Delta_i$ valuations with $v(z_i) < 0$ and at most $\Delta_i \deg_{z_i} f$ valuations with $v(f) > 0$. So
\[
|S| \leq \Delta_i + \Delta_i \deg_{z_i} f + H_{M_i/k_i}(x) + H_{M_i/k_i}(y) \leq \Delta_i (1 + \deg f) + H_{M_i/k_i}(x) + H_{M_i/k_i}(y)
\]
Now we consider
\[
x^p - y^q = 1
\]
as an $S$-unit equation in $M_i$. Because $x^p \not \in k_i$ and $y^q \not \in k_i$, we can apply Theorem \ref{tMason} resulting in
\[
H_{M_i/k_i}(x^p) \leq |S| + 2g_{M_i/k_i} - 2 \leq \Delta_i (1 + \deg f) + H_{M_i/k_i}(x) + H_{M_i/k_i}(y) + 2g_{M_i/k_i} - 2
\]
and
\[
H_{M_i/k_i}(y^q) \leq |S| + 2g_{M_i/k_i} - 2 \leq \Delta_i (1 + \deg f) + H_{M_i/k_i}(x) + H_{M_i/k_i}(y) + 2g_{M_i/k_i} - 2.
\]
Define $K_i = k_i(z_i, u)$. Then we have that $[K_i : k_i(z_i)] \leq D$. Hence
\[
H_{M_i/k_i}(x) = [M_i : K_i] H_{K_i/k_i}(x) \geq [M_i:K_i] = \Delta_i/[K_i : k_i(z_i)] \geq \Delta_i/D
\]
and similarly for $y$. This gives
\[
\frac{\Delta_i}{D}(p - 2 + q - 2) \leq (p - 2)H_{M_i/k_i}(x) + (q - 2)H_{M_i/k_i}(y) \leq 2\Delta_i (1 + \deg f) + 4g_{M_i/k_i} - 4,
\]
hence
\[
p + q - 4 \leq \frac{D}{\Delta_i}(2\Delta_i (1 + \deg f) + 4g_{M_i/k_i} - 4) \leq 2D(1 + \deg f) + \frac{D}{\Delta_i}4g_{M_i/k_i}.
\]
Recall that $\Delta_i = [M_i : k_i(z_i)]$ and that $M_i$ is the splitting field of $F$ over $k(z_i)$. So Lemma \ref{lSchmidt} gives
\[
g_{M_i/k_i} \leq (\Delta_i - 1) D \max_j \deg_{z_i}(F_j) \leq \Delta_i \cdot D \cdot (2d)^{\exp O(r)}.
\]
Combining gives
\[
p + q - 4 \leq 2d^t(1 + (2d)^{\exp O(r)}) + 4(d^t)^2 (2d)^{\exp O(r)} \leq (2d)^{\exp O(r)}
\]
and hence (\ref{resultt}).
\end{proof}

\section{Acknowledgements}
This research was done as part of the author's PhD at Leiden University. The results are based on the author's Master Thesis, which was written under the supervision of J.-H. Evertse. I thank him for his many valuable ideas.

\[
\mathbf{\bibname}
\]

\bibliography{MT}
\end{document}